\documentclass[10pt, twoside, a4paper]{article}
\usepackage{amsmath}
\usepackage{tikz}
\usepackage{graphicx}
\usepackage{subfigure}
\usepackage{calc}
\usepackage{float}
\usepackage{hyperref}

\newtheorem{remark}{Remark}
\newtheorem{conj}{Conjecture}
\usepackage{amsmath,amscd,amssymb}
\usepackage{graphics}
\newtheorem{theo}{Theorem} 
\newtheorem{lem}{Lemma}
\newtheorem{cor}{Corollary}
\newtheorem{prop}{Proposition}
\newtheorem{defi}{Definition}
\newtheorem{proof}{Proof}

\newskip\ttglue\ttglue=.5em plus.25em minus.15em
\def\firstname#1{\def\FIRSTNAME{#1}\ignorespaces}                  
\def\lastname#1{\def\LASTNAME{#1}\ignorespaces}
\def\middleinitial#1{\def\MIDDLEINI{#1}\ignorespaces}
\def\department#1{\def\DEPARTMENT{#1}\ignorespaces}
\def\institute#1{\def\INSTITUTE{#1}\ignorespaces}
\def\address#1{\def\ADDRESS{#1}\ignorespaces}
\def\country#1{\def\COUNTRY{#1}\ignorespaces}
\def\otheraffiliation#1{\def\OTHERAFFILIATION{#1}\ignorespaces}
\def\email#1{\def\EMAIL{#1}\ignorespaces}
\newcount\autcount\autcount=0
\newcount\affcount\affcount=0
\newcount\numcount\newcount\nummcount
\newcount\nummmcount\newcount\nummmmcount
\def\writename#1#2{\ \kern-1ex\hbox{
  \csname AUthor\the#1\endcsname\
  \edef\TESTSTR{}\expandafter\ifx\csname auTHor\the#1\endcsname\TESTSTR
  \else\csname auTHor\the#1\endcsname.\ \fi 
  \csname authOR\the#1\endcsname$^{\csname AFF\the#1\endcsname}$
  \expandafter\ifx\csname corr\number#1\endcsname\relax
  \else\thanks{Corresponding author.}\ \fi
  }\ifnum#1<#2, \else\ \kern-1ex\fi}
\def\writeemail#1{
  \nummcount=0\relax\nummmcount=0\relax
  \loop\ifnum\nummcount<\autcount\advance\nummcount by1\relax
    {\expandafter\ifnum\csname AFF\the\nummcount\endcsname=#1\relax
    \global\advance\nummmcount by1\fi}\repeat
  \nummcount=0\relax\nummmmcount=0\relax
  \loop\ifnum\nummcount<\autcount\advance\nummcount by1\relax
    {\expandafter\ifnum\csname AFF\the\nummcount\endcsname=#1\relax
    \global\advance\nummmmcount by1\relax\def\blank{}\expandafter
    \ifx\csname EMAIL\the\nummcount\endcsname\blank(no e-mail)
    \else\csname EMAIL\the\nummcount\endcsname
    \fi 
    \ifnum\nummmmcount<\nummmcount; \fi\fi}\repeat}
\long\def\BeginAuthorList#1\EndAuthorList{#1\relax
  \author{\vbox{\hsize=390pt\noindent\numcount=0\relax
    \loop\ifnum\numcount<\autcount\advance\numcount by1\relax
      \writename{\numcount}{\autcount}
      \repeat}\\[2mm]
    \vbox{\small\numcount=0\relax
      \loop\ifnum\numcount<\affcount\advance\numcount by1\relax
        \vbox{{\count0=\numcount\relax
          \loop\expandafter\ifnum\csname AFF\the\count0\endcsname
            <\numcount\relax\advance\count0 by1\relax\repeat
          $^{\csname AFF\the\count0\endcsname}$}
        \def\BLANK{}\expandafter\ifx\csname DEPT\the\numcount\endcsname
          \BLANK
          \else\csname DEPT\the\numcount\endcsname, \fi
        \csname INST\the\numcount\endcsname,
        \csname ADDR\the\numcount\endcsname,
        \csname COUN\the\numcount\endcsname
        \edef\TEST{}\expandafter\ifx\csname OTHE\the\numcount\endcsname
          \TEST
          .\else;\break\csname OTHE\the\numcount\endcsname.\fi}
        \vbox{\writeemail{\numcount}}
        \repeat}\\}}
\expandafter\def\csname x1\endcsname{}
\expandafter\def\csname x2\endcsname{}
\expandafter\def\csname x3\endcsname{}
\expandafter\def\csname x4\endcsname{}
\expandafter\def\csname x5\endcsname{}
\expandafter\def\csname x6\endcsname{}
\expandafter\def\csname x7\endcsname{}
\expandafter\def\csname x8\endcsname{}
\expandafter\def\csname x9\endcsname{}
\def\Author#1#2{\global\advance\autcount by1\relax#2
  \expandafter\edef\csname AUthor\the\autcount\endcsname{\FIRSTNAME}
  \expandafter\edef\csname auTHor\the\autcount\endcsname{\MIDDLEINI}
  \expandafter\edef\csname authOR\the\autcount\endcsname{\LASTNAME}
  \expandafter\edef\csname EMAIL\the\autcount\endcsname{\EMAIL}
  \let\tempera\"\def\"{\string\"}\expandafter\ifx\csname x\DEPARTMENT
    \endcsname\relax
    \global\advance\affcount by1\relax\let\"\tempera
    \expandafter\edef\csname DEPT\the\affcount\endcsname{\DEPARTMENT}
    \expandafter\edef\csname INST\the\affcount\endcsname{\INSTITUTE}
    \expandafter\edef\csname ADDR\the\affcount\endcsname{\ADDRESS}
    \expandafter\edef\csname COUN\the\affcount\endcsname{\COUNTRY}
    \expandafter\edef\csname OTHE\the\affcount\endcsname{\OTHERAFFILIATION}
    \expandafter\edef\csname AFF\the\autcount\endcsname{\the\affcount}
  \else\expandafter\edef\csname AFF\the\autcount\endcsname{\DEPARTMENT}
  \fi\let\"\tempera\ignorespaces}
\def\CorrespondingAuthor#1#2{
  \expandafter\xdef\csname corr\number#1\endcsname{cor}
  \Author#1{#2}}
\def\PaperTitle#1{\title{\bf#1}}
\def\Category#1{\ignorespaces}
\def\keywords#1{{\noindent \emph{Keywords:}
\def\BLANK{}\def\TEST{#1}\ifx\BLANK\TEST(n/a).\else#1\fi}}
\hypersetup{hidelinks}
\begin{document}
\large
\PaperTitle{Topological and Diophantine properties of  
lattice subset projections} 

\Category{(Pure) Mathematics}

\date{}

\BeginAuthorList
\Author1{
\firstname{Wayne}
    \lastname{Lawton}
    \middleinitial{M}
    \department{Department of the Theory of Functions, 
Institute of Mathematics and Computer Science}
    \institute{Siberian Federal University}
    \institute{Siberian Federal University}
    \otheraffiliation{}
   \address{Krasnoyarsk}
    \country{Russian Federation}
    \email{wlawton@gmail.com}}
\EndAuthorList
\maketitle
\thispagestyle{empty}
\begin{abstract}
\noindent For $1 \leq n < m, k := m-n,$
the Grassmannian $Gr(n,m),$ 
consisting of $n$-planes $U$ in
$\mathbb R^m,$ is a compact 
$nk$-dimensional real-analytic 
manifold
with a transitive 
$O(m)$ action and a unique $O(m)$-invariant
probability measure $\sigma_n.$ 
Let $P_U : \mathbb R^m \to U$ be
orthogonal projection and
$\psi : \mathbb Z_+ \to \mathbb R_+.$
We relate properties of
$L \subset \mathbb Z^m$ 
to those of its associated subsets 
$L_{n,\mathrm{locfin}}, L_{n,lim}, L_{n,dense}$
and
$L_{n,\psi},$ 
consisting of those
$U \in Gr(n,m)$ 
for which, respectively,
$P_UL$ is locally finite,
has $0$ as a limit point,
is dense, and satisfies 
$0 < |P_U\ell|_\infty < \psi(|\ell|_\infty)$
for infinitely many 
$\ell \in L.$
We use Baire's category theorem
to prove that $L_{n,dense}$ is residual iff
$L$ is $k$-dense: for every $k$-plane $V$ there exists a sequence
$\ell_j \in L$ such that the angle between $\ell_j$ and $V$ 
approaches $0.$
We derive an upper bound for the 
$\sigma_k$-measures of 
certain tubular neighborhoods in
$Gr(k,m),$ and use it to prove that
$\sigma_n(L_{n,lim}) = 0$ 
whenever $\sum_{\ell \in L \backslash \{0\}} |\ell|_2^{-n} < \infty.$
For non-increasing $\psi,$ we  use the Khintchine–Groshev theorem to prove that
$\sigma_n(\mathbb Z^{m}_{n,\psi}) = 0$ if
$S := \sum_{j \in \mathbb Z_+} j^{k-1} \psi(j)^n < \infty,$
and conjecture that $\sigma_n(L_{n,\psi}) = 1$ if $S =\infty.$
\end{abstract}
\noindent{\bf 2020 Mathematics Subject Classification}
11P21; 
54E52; 
11J83 
\newline 
\\
{\bf Keywords} Grassmannian; lattice subset; orthogonal projection; Baire
category; Diophantine approximation; Khintchine--Groshev theorem;
almost periodic set; Fourier quasicrystal
\tableofcontents
\section{Introduction and Preliminary Results}\label{sec1} 
Throughout  this paper we fix
integers $1 \leq n < m$ 
and $k := m-n$ and let 
$Gr(n,m)$ denote the 
Grassmannian consisting of 
$n$-dimensional subspaces (called $n$-planes) 
of $\mathbb R^m.$
For $U \in Gr(n,m),$
$P_U : \mathbb R^m \to U$ is orthogonal
projection. A subset of $U$ is called dense
if its closure equals $U,$ and locally finite
if its intersection with every bounded subset
is finite. 
\begin{defi}\label{def1.0} 
For a lattice subset 
$L \subset \mathbb Z^m$  and sequence 
$\psi : \mathbb Z_+ \to \mathbb R_+$ define associated subsets of 
$Gr(n,m):$ 
\begin{enumerate}
\item $L_{n,locfin} := 
\{U \in Gr(n,m) : P_UL \mbox{ is a locally finite subset of } U \}.$
\item $L_{n,lim} := \{U \in Gr(n,m) : 0 
\mbox{ is a limit point of } P_UL\}.$
\item $L_{n,dense} := \{U \in Gr(n,m) :
P_UL \mbox{ is a dense subset of }  U\}.$
\item $L_{n,\psi} :=$ 
set of $U \in Gr(n,m)$ such that 
$0 < |P_U\ell|_\infty < \psi(|\ell|_\infty)$
holds for infinitely many $\ell \in L.$
\end{enumerate}
\end{defi}
This paper studies relations between properties of $L$ and $\psi$ and both topological 
and metric (measure-theoretic) properties of 
sets 1-4. These relations have intrinsic interest because they generalize concepts in Oxtoby's book \cite{oxtoby} titled Measure and Category.
\S1.2 discusses how lattice subset projections 
describe spectrums of quasicrystals, 
certain almost periodic sets, and 
Fourier quasicrystals. This was our 
original motivation to study them.
\S1.4 precisely describes sets
1-3 for $L = \mathbb Z^m.$ 
\newline
\\
\S2.1 explains how
$Gr(n,m)$
is represented a homogeneous space with a transitive 
$O(m)$ action, a unique $O(m)$-invariant probability measure 
$\sigma_n,$ and an $nk$-dimensional compact real-analytic 
manifold structure.
A subset $Y$ of a topological space is called 
residual if it contains 
a countable intersection of dense open subsets.
Theorem \ref{theo2.2.11} in \S2.2 uses Baire's category theorem to
prove that $L_{n,dense}$ is residual iff
$L$ is $k$-dense: for every $k$-plane $V$ there exists a sequence
$\ell_j \in L$ such that the angle between $\ell_j$ and $V$ 
approaches $0.$
\newline
\\
When discussing any Euclidean space
$\mu$ denotes Lebesgue measure on that
space.
A measurable subset of a measure
space equipped with a positive
measure is called null if its
measure equals $0$ and conull if
the measure of its complement equals 
$0.$ 
For $\rho > 0,$ $L$ is called $\rho$-lacunary if 
$\sum_{\ell \in L \backslash \{0\}} 
| \ell |_2^{-\rho} < \infty.$
\S2.3 shows how the embedding
$U \mapsto P_U$ of $Gr(n,m)$ into
the set $P_{n,m}$ of projection matrices
in $\mathbb R^{m \times m}$ with rank $n,$ 
gives it the structure of a Riemannian manifold 
equivalent to its structure as a homogeneous space. 
It derives both lower and positive bounds for the 
$\sigma_n$-measure of geodesic balls in Grassmannians.
\S2.4 uses these bounds to prove that certain residual subsets of $L_{n,lim}$ are null if
$L$ is $nk$-lacunary. 
Theorem \ref{thm2.4.1} in \S2.5 
proves
that $L_{n,\lim}$ is $\sigma_n$-null if
$L$ is $n$-lacunary.
\newline
\\
\S3.1 imparts a manifold structure 
on $Gr(n,m),$ using an atlas of charts,
and uses the inverse function theorem to 
proves it is equivalent to the Riemannian manifold structure. 
\S3.2 uses the Cauchy-Binet formula to construct a 
collection of $\binom{m}{n}$ compact sets which 
covers  $Gr(n,m).$
Theorem \ref{thm3.3.1} in \S3.3
uses chart coordinates
and the classic  Khintchine–Groshev  theorem in Diophantine approximation
to prove that 
$L_{n,\psi}$ 
is $\sigma_n$-null if
$\psi$ is non-increasing and 
$S := \sum_{j \in \mathbb Z_+} j^{k-1} \psi(j)^n < \infty.$
Conjecture \ref{con3.3.1}
states that 
$\sigma_n(L_{n,\psi})$ is conull if  
$\psi$ is non-increasing and $S =\infty.$
\subsection{Notation}\label{subsec1.1}
$:=$ means is defined to be,
iff $:=$ if and only if. $\exists :=$ there exists.
$\forall := $ for all.
For sets $X, Y,$ 
$|X| :=$ cardinality of $X$ and 
$X \backslash Y := \{x \in X  : x \notin Y\}.$
$\mathbb Z, \mathbb Z_+,\mathbb Q, \mathbb R,
\mathbb R_+,$ and $\mathbb C$ 
are  the sets of 
integers, positive integers, rationals, reals, 
positive reals, and complex numbers. 
$\cong$ denotes isomorphism of groups.
If $V$ is a real vector space and $S \subset V,$
$\mathbb R(S) :=$ real subspace of $V$ spanned by $S,$ 
and 
$\mathbb Z(S) :=$ subgroup of $V$ generated by $S.$
For $m \in \mathbb Z_+,$ 
$\mathbb R^m$ is a real vector space
of column vectors 
$x = [x_1,...,x_m]^T,$
where $^T$ is transpose. 
It is a Banach space under the norm 
$|x|_p := (|x_1|^p+\cdots +|x_m|^p)^{1/p}, p \geq 1,$
and the max norm
$|x|_\infty :=  \max \{|x_1|, ...,|x_m| \},$
and a Euclidean space with
scalar or dot product
$x\cdot y := x_1y_1 + \cdots + x_m y_m$
and associated norm $|x|_2.$
For a subspace $U \subset \mathbb R^m,$
$U^\perp := \{x \in \mathbb R^m : 
u \in U \implies x \cdot u = 0\},$
is its orthogonal complement and 
$P_U : \mathbb R^m \to U$ is orthogonal projection.
Clearly $U^{\perp \perp} = U,$  and
$kernel \, P_U := \{x \in \mathbb R^m : P_Ux = 0\} = U^\perp.$
For $r > 0,$
$B_m(x,r) := \{y \in \mathbb R^m : |y-x|_2 < r\}$
and 
$S_m(x,r) := \{y \in \mathbb R^m : |y-x|_2 = r\}.$
For $a, b \in \mathbb Z_+$ we identify a matrix 
$M \in \mathbb R^{a \times b}$
with the linear map 
$M : \mathbb R^b \rightarrow \mathbb R^a.$
$\mathbb R^{a \times b}$ is a Euclidean space
isometric to $\mathbb R^{ab}$ under the Frobenius scalar product
$A\cdot B := trace A^TB$ and norm
$|A|_F := \sqrt {trace A^TA}.$
It is a Banach space under the operator $2$-norm 
$|A|_2 := \max_{x \in S_b(0,1)} |Ax|_2.$
These norms satisfy $|A|_2 \leq |A|_F.$
$I_m \in \mathbb Z^{m \times m}$ is the identity matrix.
Its column vectors $e_1,...,e_m$ form the standard orthonormal
basis for $\mathbb R^m.$
The general linear group
$GL(m,\mathbb R) := 
\{M \in \mathbb R^{m \times m} : \det M \neq 0 \}.$ 
$O(m) := \{M \in GL(m,\mathbb R) : 
M^TM = I_m\}$  
is its subgroup of orthogonal matrices.
The unimodular group is
$GL(m, \mathbb Z) := 
\{M \in \mathbb Z^{m \times m} : \det M = \pm 1\}.$
For functions $h, g : \mathbb R_+ \rightarrow R_+$
we write $h = O(g)$ if there exist $C, r_0 \in \mathbb R_+$
with  $h_k(r) \leq Cg(r), \, r \geq r_0.$
\subsection{Motivation}\label{subsec1.2}
%
$\Lambda \subset \mathbb R^n$ $\Lambda$ is call Delone if there exist $R_2 > R_1 > 0$ such that
$B_n(\lambda,R_1), \lambda \in \Lambda$ are pairwise disjoint, and
$B_n(\lambda,R_2), \lambda \in \Lambda$
cover $\mathbb R^n.$
Lagarias, in a landmark survey paper 
(\cite{lagarias}, p. 15  Definition 3.5)
 called $\Lambda$ Bohr almost periodic \cite{bohr} if the associated measure $\mu_\Lambda := \sum_{\lambda \in \Lambda}
\delta_\lambda,$ where $\delta_\lambda$ is point measure at $\lambda,$ is a uniformly almost periodic  distribution. 
In \cite{lawton1} we explained the spectrum 
$\Omega(\Lambda) \subset \mathbb R^n,$ and defined $\Lambda$ 
to be of toral type $m$ if 
$\mathbb R(\Omega(\Lambda)) = \mathbb R^n$
and
$\mathbb Z(\Omega(\Lambda)) \cong \mathbb Z^m.$
The rank of a matrix $M$ is
the common value of $d_r$ and $d_c,$
the dimensions of the spaces spanned by the rows and by the columns of $M.$
\begin{lem}\label{lem1.0.1} 
If $m \in \mathbb Z_+$ and 
$\Lambda$ is of toral type $m,$ 
then there exists a rank $n$
$M \in \mathbb R^{m \times n}$
whose rows are rationally independent, 
and a subset 
$L \subset \mathbb Z^m$ such that
$\mathbb Z(\Omega(\Lambda)) = M^T\mathbb Z^m$
and
$\Omega(\Lambda) = M^TL.$
$M$ and $L$ are unique up to a transformation
$M \mapsto gM$ and
$L \mapsto (g^T)^{-1} L$ where
$g \in GL(m,\mathbb Z).$
Moreover
$n \leq m,$ and 
$n = m$ iff $\mathbb Z(\Omega)$ is a lattice
subgroup of $\mathbb R^n.$
Let $U := M\mathbb R^n.$ 
Then $U \in Gr(n,m)$
and
\begin{enumerate}
\item $M^TL$ is locally finite iff $P_UL$ is locally finite.
\item $0$ is a limit point of $M^TL$ iff
$0$ is a limit point of $P_UL,$
\item $M^TL$ is dense iff $P_UL$ is dense,
\end{enumerate}
\begin{equation}\label{eq1.0.1} 
|M^TL \cap B_n(0,r)| = O(r^m) \iff |P_UL \cap B_m(0,r)| = O(r^m)
\end{equation}
\end{lem}
\begin{proof}
Choose $M$ such that the columns of $M^T$ 
form a $\mathbb Z$-basis for 
$\mathbb Z(\Omega) \subset \mathbb R^n$ 
(see Definition \ref{def1.1.1}).
Lemma \ref{lem1.1.1} implies that $M^T$ has
$m$ columns hence $M \in \mathbb R^{m \times n}$
and the rows of $M$ are rationally independent.
Then $M^T \mathbb Z^m = \mathbb Z(\Omega).$
Since $M^T \mathbb R^m = \mathbb R(\Omega) = \mathbb R^n,$
$rank M = n$ hence $n \leq m,$ and $n = m$ iff $M^T$ is invertible iff $M^T\mathbb Z^m$ is a lattice subgroup of $\mathbb R^n.$ Since
$rank M = n,$ $dim U = n$ so $U \in Gr(n,m).$
Define $N := M(M^TM)^{-1}.$
Then $NM^T$ is symmetric, idempotent and has rank $n.$ Therefore it 
is orthogonal projection onto an $n$-dimensional subspace of $\mathbb R^m.$ Since the image of $NM^T$ is contained in $U,$
$NM^T = P_U.$
Assertions 1-3 follow since 
$N : M^T\mathbb R^m \to U $ is a linear homeomorphism
so $M^TL$ and $P_UL$ share the same topological properties. 
To prove (\ref{eq1.0.1}) it suffices to prove the following 
implications for every $\ell \in L$
\begin{equation}\label{eq1.0.20} 
M^T\ell \in B_n(0,r) \implies  P_U\ell \in B_m(0,|N|_2\, r)
\end{equation}
and
\begin{equation}\label{eq1.0.3} 
P_U \ell \in  B_m(0,r) \implies M^T\ell \in B_n(0,|M^T|_2 \, r).
\end{equation}
(\ref{eq1.0.20}) follows since if $M^T\ell \in B_n(0,r)$ then 
$|P_U\ell|_2 = |NM^T\ell|_2 \leq |N|_2 |M^T\ell|_2 < |N|_2 r $ hence
$P_U\ell \in B_m(0,|N|_2\, r).$
Clearly $M^T = M^TP_U.$ 
\newline
(\ref{eq1.0.3}) follows since if $P_U \ell \in B_m(0,r)$ then
$|M^T\ell|_2 = |M^TP_U\ell|_2 \leq |M^T|_2 |P_U\ell|_2  < |M^T|_2 r$
hence $M^T\ell \in B_n(0,|M^T|_2 \, r).$
\end{proof} 
\begin{remark} 
\cite{alon} and \cite{lawton2} study a class of
$\Lambda \subset \mathbb R^n$ called Fourier quasicrystals. (\cite{lawton2}, p. 93,  Proposition 2)
proved that every such $\Lambda$ is a Bohr almost periodic set. (\ref{eq1.0.1}) is of interest because if 
$\Omega(\Lambda) = M^TL$ satisfies 
$|M^TL \cap B_n(0,r)| = O(r^m)$ then $\Lambda$ is a Fourier quasicrystal. Alon and Kummer shared their proof (yet unpublished) that every Fourier quasicrystal is of toral type.
\end{remark}
\begin{remark} 
\cite{lawton2} explains that certain Delone sets $\Lambda \subset \mathbb R^n$ that model Penrose tilings and physical quasicrystals are also of toral type and why they are Besicovitch but not Bohr almost periodic sets. It cites numerous papers by Sergii Favorov, Yves Meyer, and other researchers who studied these topics.
\end{remark}
\subsection{Free Abelian Groups}\label{sub1.3}
%
\begin{defi}\label{def1.1.1} 
A subset $\mathcal B$ of an abelian group $A$ is a $\mathbb Z$-basis if every 
$a \in A$ determines a unique finitely supported
$c _a : \mathcal B \rightarrow \mathbb Z$ such that
\begin{equation}
	a = \sum_{\beta \in \mathcal B} c_a(\beta) \, \beta.
\end{equation}
$A$ is called free if it has a $\mathbb Z$-basis. 
\end{defi}
\begin{lem}\label{lem1.1.1} 
	All $\mathbb Z$-bases of a free abelian group have the same cardinality (its rank). Free abelian groups are isomorphic iff their ranks are equal.
\end{lem}
\begin{proof}
Theorem 1.2 in \cite{hungerford} .
\end{proof}
$\mathbb Z$-bases of $\mathbb Z^m$ coincide with sets of column of elements in $GL(m,\mathbb Z).$  The next result shows 
every subgroup $C$ of $\mathbb Z^m$ is free with rank $\leq m.$
\begin{lem}\label{lem1.1.2}  
Either $C = \{0\}$ or there exists an $\mathbb Z$- basis
$\{b_1,...,b_m\}$ of $\mathbb Z^m,$ 
$1 \leq p \leq m,$ and
$k_1,...,k_p \in \mathbb Z_+$
such that
$\{k_1b_1,...,k_pb_p\}$ is a 
$\mathbb Z$-basis
for $C.$
\end{lem}
\begin{proof}
Theorem 1.6 in \cite{hungerford}.
\end{proof}
Richard Dedekind first proved that subgroups of free abelian 
groups are free.  A constructive proof is obtained by computing 
the Smith normal form for integer matrices \cite{smith}, \cite{stanley}.
\begin{defi}\label{def1.1.2} 
For subspaces $U$ of $\mathbb R^m$ define
$r(U) := rank (U \cap \mathbb Z^m),$
$Rat(n,m) := \{U \in Gr(n,m) : r(U) = n\},$
$Irr(n,m) := \{U \in Gr(n,m) : r(U) = 0\}.$
\end{defi}
\begin{lem}\label{lem1.1.33} 
$Rat(n,m) = 
\{U \in Gr(n,m) : P_U \in \mathbb Q^{m \times m} \}.$ 
\newline
$Rat(n,m)^\perp = Rat(m-n,m).$ 
$Rat(n,m)$ is countable and dense  
\end{lem}
\begin{proof} If $U \in Rat(n,m)$ then $r(U) = n$ so $U$ has an $\mathbb R$-basis 
$\{b_1,...,b_n\}  \subset \mathbb Z^m.$
Construct an orthogonal basis for $U$ using the Gram-Schmidt orthogonalization algorithm
applied to  the sequence  $b_1,...,b_n,$ 
so $m_1 = b_1$ and
$$ 
m_{j} = b_j - \sum_{i=1}^{j-1}  
 \frac{m_i \cdot b_j}{m_i \cdot m_i}\, m_i, \ \ j = 2,...,n.
$$
Therefore $m_1,...,m_n \in \mathbb Q^m.$
Define matrices
$M := [m_1 \, m_2 \, \cdots \, m_n] \in \mathbb Q^{m \times n}$ and $D := diag(d_1,...,d_n) \in \mathbb Q^{n \times n}$ where
$d_i := m_i \cdot m_i, \, i = 1,...,n.$
Then
$P_{U} = MD^{-1}M^T \in \mathbb Q^{m \times m}.$ 
Conversely, if $P_U \in \mathbb Q^{m \times m}$ and 
$d$ is the least common multiple of the denominators of 
the entries of $P_U,$ then 
$d P_U \in \mathbb Z^{m \times m}.$ 
Since rank $dP_U = $dim $U = n$ there exist $n$ columns of $dP_U$ that are linearly independent and therefore form 
an $\mathbb R$- basis of $U$ 
consisting of integer vectors. Therefore $r(U) = n$ so $U \in Rat(n,m).$
The second assertion follows since
$P_U + P_{U^\perp} = I_m.$ $Rat(n,m)$ is countable since  
$\mathbb Q^{m \times m}$ is countable, and dense since  $\mathbb Q$ is dense in $\mathbb R.$
\end{proof}
\begin{defi}\label{def1.1.3} 
Subgroup $C, D \subset \mathbb Z^m$ are complements if 
$\mathbb Z^m = C \oplus D.$ 
\end{defi}
\begin{prop}\label{prop1.1.1} 
A subgroup $C \subset \mathbb Z^m$ has a complement 
iff $C = \mathbb R(C) \cap \mathbb Z^m.$
\end{prop}
\begin{proof}
$\implies$ Let $U := \mathbb R(C).$ If $C$ has complement $D,$
let $V := \mathbb R(D)$ and 
$\ell \in U \cap \mathbb Z^m.$ It suffices to prove that $\ell \in C.$
Since $C+D = \mathbb Z^m,$
$\ell = c + d$ where $c \in C$ and $d \in D$
so $\ell - c = d$ belongs to $U \cap V,$
hence it suffices to prove that
$U \cap V  = \{0\}.$
Clearly
$U+V = \mathbb R^m,$
$dim \, U \leq rank\, C,$
$dim \, V \leq rank\, D.$
Since
$rank \, C + rank \, D = m,$
$dim \, U + \dim \, V \leq m,$
hence
$U \cap V  = \{0\}.$
\newline
$\impliedby$
If $C = \{0\},$ then
$\mathbb Z^m$ is a complement of $C,$ so assume
$p := rank \, C \geq 1.$ Let
$\mathcal B := \{b_1,...,b_m\},$ resp. 
$\mathcal B_C := \{c_1b_1,...,c_nb_p\}$
be the $\mathbb Z$-basis for $\mathbb Z^m,$ resp. $C$ given by Lemma \ref{lem1.1.2}. 
Then for $1 \leq j \leq p,$ 
$b_j = c_j^{-1}(c_jb_j) \in \mathbb R(C) \cap \mathbb Z^m = C.$ Therefore 
$C = \mathbb 
Z(\{c_1b_1,...,c_p b_p\})
= \mathbb Z(\{b_1,...,b_p\})$
hence
$\{b_1,...,b_p\}$ is a 
$\mathbb Z$ basis for $C$ and 
$D := \mathbb 
Z(\{b_{p+1},...,b_m\})$ 
is a complement of $C.$
We also note that since
$b_1,..,b_p$ are rationally independent, each $c_j = \pm 1.$
\end{proof}
\subsection{Projections of $\mathbb Z^m$}\label{sub1.4}
%
Let $U \in Gr(n,m)$ and define the closed subgroup
$G(U) = \overline {P_U(\mathbb Z^m)}.$ 
\begin{lem}\label{lem1.2.1} 
$G(U) \cong \mathbb R^a \oplus \mathbb Z^b$ 
where $0 \leq a,$ $0 \leq b$ and $a+b = n.$ 
\end{lem}
\begin{proof}
Since $G(U)$ is a  closed subgroup of $U$ and 
$U \cong \mathbb R^n,$ 
(\cite{bourbaki}, TG VII. 2. Sous-groupes fermés 
de $\mathbb R^n$) implies that
$G(U) \cong \mathbb R^a \oplus \mathbb Z^b $ 
where $a+b \leq n.$  
Therefore   
$\mathbb R(G(U)) \cong \mathbb R^{a+b}.$
Since $P_U$ is continuous and linear
$\mathbb R(G(U)) =
\mathbb R(\overline {P_U\mathbb Z^m}) = 
\overline {\mathbb R(P_U\mathbb Z^m)}.$
Since $P_U$ is linear
$\mathbb R(P_U\mathbb Z^m) = 
P_U(\mathbb R(\mathbb Z^m)) = 
P_U(\mathbb R^m) = U.$
Therefore
$\mathbb R(G(U)) = U,$
hence $a+b = n.$
\end{proof}
\begin{lem}\label{lem1.2.2} 
$r(U) = b.$ Consequently
for $L = \mathbb Z^m,$ 
we identify 
$L_{n,\mathrm{locfin}}$ 
with the set $Rat(n,m)$ of
rational 
$n$-planes
and
$L_{n,dense}$ with the
set $Irr(n,m).$
\end{lem}
\begin{proof}
Define $H(U) := \{u \in U : u \cdot G(U) 
\subset \mathbb Z\}.$
Lemma \ref{lem1.2.1} implies $G(U) \cong \mathbb R^a \oplus \mathbb Z^b,$ 
hence $H(U) \cong \mathbb Z^b.$
Clearly $P_U^T = P_U,$ and for $u \in U,$ $P_U u = u.$ Therefore
$$u \cdot G(U) = u \cdot \overline {P_U\mathbb Z^m} = 
\overline {u \cdot P_U \mathbb Z^m} = \overline {(P_Uu) \cdot \mathbb Z^m} 
= \overline {u \cdot \mathbb Z^m}$$
hence
$H(U) = U \cap \mathbb Z^m.$ 
The first assertion then follows
since rank $H(U) = b.$ 
Since $a = n - b,$ Lemma \ref{lem1.2.1} 
gives $G(U) = U$ when $r(U) = 0,$ 
so $P_U\mathbb Z^m$ is dense.
\end{proof}
\section{Topological and Metric Properties of $L_{n,dense}$}\label{sec2} 
\subsection{Grassmannian Geometry}\label{sub2.1} 
%
$O(m)$ acts on $\mathbb R^m$ and maps $n$-planes to $n$-planes so acts on $Gr(n,m).$
Let $U \in Gr(n,m)$ and $S(U)$ be the stabilizer subgroup of those $g \in O(m)$ with $gU = U.$
Since the action of $O(m)$ on 
$Gr(n,m)$ is transitive, $Gr(n,m)$ is represented by the homogeneous space
$O(m)/S(U).$ These representations
are equivalent because stability groups
are conjugate.
Warner (\cite{warner}, p. 129, Exercise f) proves that
$O(m)/S(U)$ is a differentiable manifold.
Helgason (\cite{helgason}, p. 123, Theorem 4.2) proves the stronger result that
$O(m)/S(U)$ is a real-analytic manifold.
Since $U^\perp \in Gr(k,m),$  $Gr(k,m)$ is represented as a $O(m)/S(U^\perp).$ 
It is easy to show that $S(U^\perp)$ is conjugate to 
$S(U).$ Therefore the Grassmannians are $Gr(n,m)$ and $Gr(k,m)$ are equivalent as homogeneous spaces. 
$O(m)$ has a unique probability measure that is invariant  under left translation called Haar measure, 
after Alfred Haar \cite{haar} who invented it in 1933. 
Alfsen \cite{alfsen} gives an accessible explanation.
Andre Weil (\cite{weil2}, p. 42-45, \S 9. Measures dans les espaces homog\`{e}nes)  derived precise conditions
for any homogeneous space $G/H$ to have a $G$-invariant measure and showed that if one exists it equals the pushforward of Haar measure on $G$ so is unique up to a constan multiple. He proved that compactness of $H$ is sufficient (but not neccessary) for existence.
It follows that $Gr(n,m)$ has a unique $O(m)$-invariant probability measure. We denote it by $\sigma_n.$
Nachbin (\cite{nachbin}, p. 138, Theorem 1 (Weil))
gives an accessible explanation of Weil's results.
\begin{lem}\label{lem2.1.1} 
The map $U \mapsto U^\perp$ is a measure preserving bijection from the measure space $(Gr(n,m),\sigma_n)$ to the  measure space $(Gr(k,m), \sigma_k).$
\end{lem}
\begin{proof}
Let $\nu$ denote Haar probability measure on $O(m),$
$B \subset Gr(n,m)$ a
a Borel set, and $U \in Gr(n,m).$ 
Then  $U^\perp \in Gr(m-n,m)$ and since $\sigma_n$ is the pishforward of Haar probability measure
$$\sigma_{k}(B^\perp) = \nu(\{g \in O(m) : gU^\perp \in B^\perp\}) = 
\nu(\{g \in O(m) : (gU^\perp)^\perp \in B\})
$$
$$ = \nu(\{g \in O(m) : gU \in B\}) = \sigma_n(B).$$
\end{proof}
For $i, j \in \mathbb Z,$ $\delta_{i,j} := 1$ if $i = j$ and
$\delta_{i,j} := 0$ if $i \neq j.$
%
%
\begin{lem}\label{lem2.1.3} 
If $\{U, V\} \subset Gr(n,m),$ then there exists 
a unique sequence of angles (called canonical or principal)
$0 \leq \theta_1 \leq \cdots \leq \theta_n  \leq \pi/2,$
and  orthonormal bases 
$\{u_1,...,u_n\}$ of $U$  
and
$\{v_1,...,v_n\}$ of $V$ 
satisfying
\begin{equation}\label{eq2.1.1}
u_i \cdot v_j = \delta_{i,j}\, \cos \theta_i.
\end{equation}
The angles are symmetric maps
$\theta_j : Gr(n,m) \times Gr(n,m) \rightarrow 
[0,\pi/2].$
The sequence of nonzero canonical angles for $\{U,V\}$
equals the sequence of nonzero canonical angles for 
$\{U^\perp, V^\perp\}.$
\end{lem}
\begin{proof}
The existence and uniqueness of canonical angles 
and the existence of associated orthonormal bases satisfying
(\ref{eq2.1.1}) was first proved by Camille Jordan in 1875 
\cite{jordan}; see also (\cite{bjorck}, p. 582, Theorem 1).
They are the foundation of canonical correlation 
analysis in statistics (\cite{johnson}, Chapter 10). The symmetry
property holds since the sequence of canonical angles only depends on the set 
$\{U, V\}$ not the ordering of $U$ and $V.$
To prove the last assertion define
$d := dim U \cap V$ and
$d^\prime = dim U^\perp \cap V^\perp.$
Then $d = n$ iff $U = V$ iff 
$U^\perp = V^\perp$ iff $d^\prime = m-n$ and then the set of nonzero canonical angles of both 
$\{U, V\}$ and
$\{U^\perp, V^\perp\}$ is empty.
Otherwise $d < n$ and the sequence of nonzero canonical angles for $\{U,V\}$
$0 < \theta_{1+d} \leq \cdots \leq \theta_n \leq \pi/2$ 
has length $n-d$  and the sequence of nonzero canonical angles for
$\{U^\perp, V^\perp\}$ has length $m-n-d^\prime.$
For
$d+1 \leq j \leq n$ 
define 
$u_j^\prime := (v_j - u_j \cos \theta_j) / \sin \theta_j $
and
$v_j^\prime := -(u_j - v_j \cos \theta_j) / \sin \theta_j.$
Then 
$\{u_j^\prime\}$ is an orthonormal basis for a subspace
$U_1 \subset U^\perp,$
$\{v_j^\prime\}$ is an orthonormal basis for a subspace
$V_1 \subset V^\perp,$
and
$u_i^\prime \cdot v_j^\prime = \delta_{i,j} \, \cos \theta_i$
hence 
$U_1 \cap V_1 = \{0\}$
and
$d^\prime \leq m-n- dim U_1 = m-n - (n-d) = m-2n+d$
so the sequence of nonzero canonical
angles for $\{U^\perp, V^\perp\}$ 
has length
$m-n-d^\prime \geq  n-d.$ 
Replacing $U$ with $U^\perp$ and $V$ with $V^\perp$
gives
$n-d \geq m-n-d^\prime$
hence $m-n-d^\prime = n - d.$
Since $m-n-d^\prime$ equals the number of nonzero canonical angles for $\{U^\perp, V^\perp\},$ 
these coincide with $\theta_{d+1},...,\theta_n$
for $\{U, V\}.$ 
\end{proof}
For $u,v \in \mathbb R^m \backslash \{0\}$ define 
\begin{equation}\label{eq2.1.4}
	\theta(u,v) := 
\arccos \left(\frac{u}{|u|_2} 
\cdot \frac{v}{|v|_2} \right)
 \in [0,\pi]. 
\end{equation}
\begin{lem}\label{lem2.1.4} 
$\theta$ is a metric on $S_m(0,1).$
\end{lem}
\begin{proof}
Let $u, v, w \in S_m(0,1).$
Then 
$\theta(u,v) = \theta(v,u)$ 
and $\theta(u,v) = 0$ iff $u = v.$
It suffices to prove the triangle inequality
$\theta(u,v) \leq \theta(u,w)
+ \theta(w,v).$
Define $P := \mathbb R(\{u, v, w\})$
and $d := dim P.$ Then $1 \leq d \leq 3.$
The triangle inequality is directly computable if
$d < 3.$  If $d = 3$ then
$S := P \cap S_m(0,1)$ is the 
two-dimensional unit sphere in $P.$ 
(\cite{carmo}, p. 246, Example 3) proves 
$\theta(u,v)$ is the geodesic distance between
$u$ and $v$ on $S$ so $\theta$ satisfies the triangle inequality.
\end{proof}
For $u \in \mathbb R^m \backslash \{0\}$ and  $V \in Gr(n,m)$ define
\begin{equation}\label{eq2.1.5}
	\theta(u,V) := \min_{v \in V \backslash \{0\}} \theta(u,v)
\in [0,\pi/2].
\end{equation}
By Lemma \ref{lem2.1.3} if $U, V \in Gr(n,m),$ and
$u \in U$ and $v \in V$ are unit vectors, then there
exist unit vectors $a, b \in \mathbb R^n$ such that 
$u = \sum_{j = 1}^n a_j\, u_j$
and
$v = \sum_{j = 1}^n b_j\, v_j$
where $\{u_1,...,u_n\},$ resp. $\{v_1,...,v_n\}$ is the orthonormal basis
for $U,$ resp. $V.$
Then (\ref{eq2.1.1})  implies
$u \cdot v = \sum_{j = 1}^n a_j \, b_j \cos \theta_j$ 
where $\theta_1,...,\theta_n$ are the canonical angles associated with $\{U,V\}.$
The method of Lagrange multipliers
applied to the definition of $\theta(u,V)$ in
(\ref{eq2.1.5}) implies that
$$
\theta(u,V) =  
\arccos \sqrt {
\sum_{j=1}^n a_j^2 (\cos \theta_j)^2.
}
 $$
Then the inequalities $0 \leq \theta_1 \leq \cdots \leq \theta_n  \leq \pi/2$ imply that
\begin{equation}\label{eq2.1.6}
\max_{u \in U \backslash \{0\}} \theta(u,V) 
= \theta(u_n,V) = \theta_n(U,V).
\end{equation}
\begin{lem}\label{lem2.1.5} 
The map $\theta_n : Gr(n,m) \times Gr(n,m) \rightarrow [0,\pi/2],$
where $\theta_n$ is the largest canonical angle in 
Lemma \ref{lem2.1.3}, is a metric.
\end{lem}
\begin{proof}
For $U, V \in Gr(n,m),$ 
$\theta_n(U,V) \geq 0$ 
since canonical angles are 
in the closed interval $[0,\pi/2].$
$\theta_n(U,V) = \theta_n(V,U)$ since 
canonical angles only depend on the set $\{U, V\},$
and $\{U,V\} = \{V,U\}.$
If $\theta_n(U,V) = 0,$ then
$\theta_j(U,V) = 0$ for all $1 \leq j \leq n.$
Then (\ref{eq2.1.1}) implies that 
$u_j \cdot v_j = 1$ 
and hence $u_j = v_j$ for $1 \leq j \leq n,$
hence $U = V.$
It suffices to prove 
the triangle inequality.
Let $U, V, W \in Gr(n,m)$ 
and $u \in U \cap S_m(0,1).$
(\ref{eq2.1.6}) and the definition of $\theta(u,V)$ in 
(\ref{eq2.1.5}) imply there exist $
w \in W \cap S_m(0,1), v \in V \cap S_m(0,1)$ with
$\theta(u,w) \leq \theta_n(U,W)$
and
$\theta(w,v) \leq \theta_n(W,V).$
(\ref{eq2.1.5}) and the triangle inequality give
$
\theta(u,V) \leq \theta(u,v)
\leq \theta_n(U,W) + \theta_n(W,V),
$
hence (\ref{eq2.1.5}) gives
$
\theta_n(U,V) = 
\max_{u \in U \cap S_m(0,1)}
\theta(u,V)
\leq \theta_n(U,W) + \theta_n(W,V).
$
\end{proof}
\begin{remark}\label{rem2.1.3} 
$\theta_n(U,V)$ is the
Hausdorff distance, defined in
(\cite{munkres}, p. 279, Exercise 7),
 between  the closed subsets
$U \cap S_m(0,1)$ and
$V \cap S_m(0,1)$ of the 
metric space $(S_m(0,1), \theta.$) 
\end{remark}
Golub and Van Loan
(\cite{golub}, p 82, Equation 2.5.1) define a 
metric on equidimensional subspaces $U,V$ of
$\mathbb R^m$ by  
$dist(U,V) := |P_U-P_V|_2$ 
and assert
(\cite{golub}, p. 331, Algorithm 6.4.3)
it equals the sin of the largest principal angle.
We state and prove their assertion.
\begin{lem}\label{lem2.1.6} 
If $U, V \in Gr(n,m)$ then
$\sin \, \theta_n(U,V) = |P_U - P_V|_2$
\end{lem}
\begin{proof} Let 
$\{u_1,...,u_n\},$ resp. $\{v_1,...,v_n\}$ be the
orthonormal basis for $U,$ resp. $V$ and 
$0 \leq \theta_1 \leq \cdots \leq \theta_n$ 
be the canonical angles in Lemma \ref{lem2.1.3}.
Define
$W := U + V.$ 
Since the restrictions
$P_U|_{W^\perp} = 0$ and $P_V|_{W^\perp} = 0,$
$$
	|P_U-P_V|_2 = \max_{w \in W \cap S_m(0,1)} |P_Uw - P_Vw|_2.
$$
(\ref{eq2.1.1}) implies that every 
$w \in W \cap S_m(0,1)$ admits the representation
$$
w = \sum_{j = 1}^n (a_j u_j + b_jv_j)
\text{ where }
|w|_2^2 = \sum_{j = 1}^n 
(a_j^2 + b_j^2 + 2a_jb_j \cos \theta_j) = 1.
$$
A direct computation gives
$$
|(P_U-P_V)\,w|_2^2 = \sum_{j = 1}^n 
(a_j^2 + b_j^2 + 2a_jb_j \cos \theta_j) \, \sin^2 \theta_j,
$$
which has the maximum value $\sin^2 \theta_n$
when $w = (u_n + v_n)/\sqrt {2 + 2 \cos \theta_n}.$
\end{proof}
\begin{lem}\label{lem2.1.7} 
$U \rightarrow U^\perp$ is an isometry 
from the metric space $(Gr(n, m), \theta_n)$ to 
$(Gr(m-n, m), \theta_{m-n}).$
\end{lem}
\begin{proof}
For $U, V \in Gr(n,m),$ 
$P_{U^\perp} = I_m - P_{U}$
and
$P_{V^\perp} = I_m - P_{V}.$
Lemma \ref{lem2.1.6} gives
$\sin \theta_n(U,V) = |P_U-P_V|_2 = 
|P_{U^\perp}-P_{V^\perp}|_2 = \sin \theta_{m-n}(U^\perp,V^\perp).$ Therefore since
$\sin : [0,\pi/2] \rightarrow [0,1]$ is
injective 
$\theta_n(U,V) = 
\theta_{m-n}(U^\perp,V^\perp).$
\end{proof}
The open balls
$$
A_k(V,\alpha) := 
\{W \in Gr(k,m) : \theta_k(V,W) < \alpha \}, \ 
V \in Gr(k,m), \, \alpha \in \mathbb R_+
$$
form a basis for the topology of
$Gr(k,m).$ $A_k(V,\alpha) = Gr(k,m)$ iff 
$\alpha > \pi/2.$
\begin{defi}\label{def2.0.1} 
For nonempty open $O \subset Gr(k,m)$ define
\begin{equation}\label{eq2.0.1}
C(O) := 
\bigcup_{V \in O} V \backslash \{0\}
\end{equation}
\end{defi}
\begin{lem}\label{lem2.1.8} 
For every $V \in Gr(k,m)$ and 
$\alpha \in \mathbb R_+$
\begin{equation}\label{eq2.1.7}
	C(A_k(V,\alpha)) = 
\{w \in \mathbb R^m \backslash \{0\} : \theta(w,V) < \alpha \}
\end{equation}
\end{lem}
\begin{proof}
If $w \in C(A_k(V,\alpha))$ then
Definition \ref{def2.0.1} implies that there exists
$W \in A_k(V,\alpha)$ with $w \in W$ so
$$
\theta(w,V) \leq \max_{x \in W \cap S_m(0,1)} \theta(x,V) = 
\theta_k(W,V) < \alpha.
$$  
To prove the reverse inclusion
assume that
$w \in \mathbb R^m \backslash \{0\}$  and
$\theta(w,V) < \alpha.$
It suffices to prove that $w \in C(A_k(V,\alpha)).$
This clearly holds in $w \in V$ so we may assume
$w \notin V$ and
define orthogonal subspaces 
$V_1 := P_V(w\mathbb R)$
and
$V_2 = V \cap V_1^{\perp}$
so $V = V_1 \oplus V_2.$
Define
$W :=  V_2 \oplus w\mathbb R.$
Then 
$\theta_k(W,V) = \theta(w,V) < \alpha$
so $W \in A_k(V,\alpha)$ and since
$w \in W,$ $w \in C(A_k(V,\alpha)).$
\end{proof}
\begin{prop}\label{prop2.1.2} 
If $L \subset \mathbb Z^m$ is not
$k$-dense, then there exist 
$V \in Gr(k,m)$ and 
$\alpha \in (0,\pi/2]$ 
with
$L \cap C(A_k(V,\alpha)) = \emptyset.$
For  $U := V^\perp$ 
\begin{equation}\label{eq2.1.8}
	|B_m(0,r) \cap P_UL| = O((r\csc \alpha)^m), \ \ r \in \mathbb R_+.
\end{equation}
$A_k(V,\alpha)^\perp \subset L_{n,locfin}.$ 
\end{prop}
\begin{proof}
If $L$ is not $k$-dense then 
Definition \ref{def2.0.1} implies that there exists a nonempty open $O \subset Gr(k,m)$ with
$L \cap C(O) = \emptyset.$ Since the open balls form a basis for the topology on $Gr(k,m),$ there exist
$V \in Gr(k,m)$ and $\alpha \in \mathbb R_+$ with
$A_k(V,\alpha) \subset O.$ Then 
Definition \ref{def2.0.1} implies that 
$C(A_k(V,\alpha)) \subset C(O),$ hence
$L \cap C(A_k(V,\alpha)) = \emptyset.$ 
Then $\alpha < \pi/2$ since 
$\alpha \geq \pi/2$ implies that
$A_k(V,\alpha) = Gr(k,m).$ 
This proves  the first assertion.
For every $\ell \in L \backslash \{0\}$ Lemma \ref{lem2.1.8} implies that
$\theta(\ell,V) \geq \alpha.$
Therefore
$$
|\ell|_2 = 
|P_U\ell|_2 \, \csc \theta(\ell,V) 
\leq |P_U\ell|_2 \csc \alpha,
$$
hence
$$
B_m(0,r) \cap P_UL \subset 
B_m(0,r\csc \alpha) \cap \mathbb Z^m, \ \ 
r \in \mathbb R_+,
$$
hence
$$
|B_m(0,r) \cap P_UL| \leq 
|B_m(0,r\csc \alpha) \cap \mathbb Z^m| = 
O((r\csc \alpha)^m).
$$
This proves the second assertion.
If  $W \in  A_k(V,\alpha)$ then
$\beta := \alpha - \theta_k(W,V) > 0$
and $A_k(W,\beta) \subset A_k(V,\alpha),$
hence
$L \cap C(A_k(W,\beta)) = \emptyset.$ 
Therefore (\ref{eq2.1.7}) holds for
$U$ replaced by $W^\perp$ and $\alpha$ replaced by $\beta,$ so $P_{W^\perp}L$ is locally finite.
Therefore 
$A_k(V,\alpha)^\perp \subset L_{n,locfin}.$
This proves the third assertion.
\end{proof}
We note that  Lemma \ref{lem2.1.7} implies that
$A_k(V,\alpha)^\perp$ is open so has positive measure. Therefore if $L$ is not $k$-dense, then
$L_{n,lim}$ misses a nonempty open set hence it
is neither residual nor conull.
\subsection{Application of Baire's Category Theorem}\label{sub2.2}
%
A topological space  $X$ is called a Baire space if every residual subset of $X$ is dense in $X.$ 
In 1899 Ren\'e Baire \cite{baire} showed that every Euclidean space has this property. 
It was shown later that locally compact spaces and complete metric spaces  
(\cite{dugundji} p. 249, Theorem 10.1), (\cite{kelly}, p. 200, Theorem 34), 
(\cite{munkres}, p. 296, Theorem 48.2), and other spaces \cite{engelking} 
are Baire spaces.
\begin{lem}\label{lem2.2.1} 
For $\ell \in \mathbb Z^m 
\backslash \{0\},$ the set
$S_{k,\ell} := 
\{V \in Gr(k,m) : \ell \in V \}$ 
is a closed subset isometric to $Gr(k-1,m-1),$ 
so it has codimension $n$ and is nowhere dense
and null.
$Irr(k,m)$ is residual and conull.
\end{lem}
\begin{proof}
The inequality
$|\theta(\ell,V) - \theta(\ell,W)| \leq \theta_k(V,W)$
implies $\theta(\ell,V)$ is a continuous function
of $V.$ Since $S_{k,\ell} = \{V \in Gr(k,m) : \theta(\ell,V) = 0\},$ $S_{k,\ell}$ is closed. 
$V, W \in S_{k,\ell}$ iff 
$V = (\mathbb R \ell) \oplus V_1$
and
$W = (\mathbb R \ell) \oplus W_1$
where
$V_1$ and $W_1$ are $(k-1)$ dimensional subspaces of
the $(m-1)$-dimensional subspace
$(\mathbb R \ell)^\perp$ which is isometric with 
$\mathbb R^{m-1}.$
So $V_1$ and $W_1$ are represented by elements in 
$Gr(k-1,m-1)$
and $\theta_{k-1}(V_1,W_1) = \theta_k(V,W).$
This shows that $S_{k,\ell}$ is an isometric embedding of
$Gr(k-1,m-1)$ into $Gr(k,m)$ hence $S_{k,\ell}$ has
dimension $n(k-1)$ so has codimension $n.$
Therefore $S_{k,\ell}$ is nowhere dense and nul.
The second assertion then follows since
$$
Irr(k,m) = \bigcap_{\ell \in \mathbb Z^m \backslash \{0\}} \, Gr(k,m) \backslash S_{k,\ell}.
$$
\end{proof}
\begin{lem}\label{2.2.2} 
If $L \subset \mathbb Z^m$ is $k$-dense, then $L+p$ is $k$-dense for every $p \in \mathbb Z^m.$
\end{lem}
\begin{proof}
Let $p \in \mathbb Z^m$ and $O \subset Gr(k,m)$ be nonempty and open. 
Since $Irr(k,m)$ is a dense subset of $Gr(k,m),$ there exists $V \in Irr(k,m)$ and $\alpha \in (0,\pi/2)$ such that 
$A_k(V,\alpha) \subset O.$
It suffices to construct 
$\ell \in L \cap C(A_k(V,\alpha)).$
Since $L$ is $k$-dense there exists 
$\ell_1 \in L \cap C(A_k(V,\alpha)).$ 
Lemma \ref{lem2.1.8} and the fact that $\ell_1 \notin V$ imply that
$\theta(\ell_1,V) \in (0,\alpha).$ 
Then there exists $\ell_2 \in L \cap 
C(A_k(V,\theta(\ell_1,V)),$ hence
$\theta(\ell_2,V) < \theta(\ell_1,V).$
Continuing in this way gives a sequence 
$\ell_j \in L \backslash \{0\}$ of distinct terms 
with $\theta(\ell_j,V)$ strictly decreasing and 
$\theta(\ell_1,V) < \alpha.$
Since the sequence $\ell_j$ has distinct terms and 
$$
|B_m(0,r) \cap L| \leq 
| \{\ell \in \mathbb Z^m : |\ell|_\infty < r\} | \leq (2r-1)^m =  O(r^m), \ r \in \mathbb R_+,$$
$|\ell_j|_2 \rightarrow \infty.$ Therefore
$$
 \limsup_{j \rightarrow \infty} \theta(\ell_j+p,V) =  
\limsup_{j \rightarrow \infty} \theta\left(\frac{\ell_j}{|\ell_j|_2}+
\frac{p}{|\ell_j|_2},V\right)$$ 
$$
= \limsup_{j \rightarrow \infty}
\theta\left(\frac{\ell_j}{|\ell_j|_2},V\right)
=\limsup_{j \rightarrow \infty}
\theta(\ell_j,V) 
\leq \theta(\ell_1,V) 
< \alpha,
$$
hence there exists $i \in Z_+$ such that
$\theta(\ell_i+p,V) < \alpha.$ Then 
Lemma \ref{lem2.1.8} implies that
$\ell_i+p \in C(A_k(V,\alpha))$ which concludes the proof.
\end{proof}
\begin{lem}\label{lem2.2.3} 
If $L$ is $k$-dense then there exists a sequence $V_i \in Gr(k,m)$ with $\{V_i : i \in \mathbb Z_+\}$ dense in $Gr(k,m),$
and a sequence $\ell_i \in L\backslash \{0\}$ with distinct terms such that 
$\ell_i \in V_i, i \in \mathbb Z_+.$
\end{lem}
\begin{proof}
The set
$Rat(k,m)$ from Definition \ref{def1.1.2}
dense and countable so has an enumeration 
$W_a, a \in \mathbb Z_+.$
Therefore the set of balls
$\{A_k(W_a,1/b) : a, b \in \mathbb Z_+\}$ has an enumeration 
$\{O_i : i \in \mathbb Z_+\}.$
Since $L$ is $k$-dense 
there exists $\ell_1 \in L \cap C(O_1)$ hence
there exists $V_1 \in O_1$ and $\ell_1 \in V_1.$ Lemma \ref{lem2.2.1} implies that 
$O_2 \backslash S_{k,\ell_1}$ is nonempty and open so there exists
$\ell_2 \in L \cap 
C(O_2 \backslash S_{k,\ell_1})$
and hence
$V_2 \in O_2 \backslash S_{k,\ell_1}$
with $\ell_2 \in V_2.$ Since 
$\ell_1 \in S_{k,\ell_1},$ 
$\ell_2 \neq \ell_1.$ Proceeding by induction gives a sequence $V_j \in O_j$
and a sequence $\ell_j \in L \cap V_j$ with 
distinct terms. Choose $V \in Gr(k,m)$ and $\alpha \in (0, \pi/2].$ To show that $\{V_i : i \in Z_+\}$ is dense it suffices to show that there exist $a, b \in \mathbb Z_+$ with 
$A_k(W_a,1/b) \subset A_k(V,\alpha).$
Since $\{W_a : a \in \mathbb Z\}$ is dense there exists $a \in \mathbb Z_+$ with $W_a \in A_k(V,\alpha).$ Then
$\theta_k(W_a,V) < \alpha$ so
$A_k(W_a,1/b) \subset A_k(V,\alpha)$
whenever $1/b < \alpha - \theta_k(W_a,V).$
\end{proof}
The property that the sequence $\ell_i$ has distinct terms
is required to prove Lemma \ref{lem2.3.1}. It is not required
to prove the next result.
\begin{prop}\label{prop2.2.1} 
$L \subset \mathbb Z^m$ is $k$-dense iff
$L_{n,lim}$ is residual.
\end{prop}
\begin{proof} 
$\impliedby$ follows from Proposition
\ref{prop2.1.2}.
\newline
$\implies$ Let $\ell_i$ and $V_i$ be the sequences constructed
 in Lemma \ref{lem2.2.3} and
$\alpha_{i,j} := (j \, |\ell_i|_2)^{-1}, 
i,j \in \mathbb Z_+.$ 
The set $\bigcup_{i = 1}^\infty  
A_k(V_i, \alpha_{i,j})$ is dense because
it contains the dense set 
$\{V_i : i \in \mathbb Z_+\},$ 
and open because it is a union of 
open sets. 
Lemma \ref{lem2.2.1} implies 
$Irr(k,m)$ is residual hence
\begin{equation}\label{eq2.2.1}
	R := Irr(k,m) \cap 
\bigcap_{j = 1}^\infty 
\bigcup_{i = 1}^\infty  
A_k(V_i, \alpha_{i,j})
\end{equation}
is residual.
Lemma \ref{lem2.1.7} implies that
$R^\perp$ is also residual. 
Let $V \in R$ and $U := V^\perp.$
Since $V \in Irr(k,m),$ 
$V \cap \mathbb Z^m = \{0\}$ hence 
$0 \notin P_U(L \backslash \{0\}).$
For every $j \in \mathbb Z_+$ there exists
$i \in \mathbb Z_+$ 
with
$V \in A_k(V_i,\alpha_{i,j}),$ 
hence 
$V_i \in A_k(V,\alpha_{i,j}),$ 
and then since $\ell_i \in V_i$ 
it follows that
$\ell_i \in C(A_k(V,\alpha_{i,j}))$ so
Lemma \ref{lem2.1.8} implies that
$\theta(\ell_i, V) < \alpha_{i,j}.$
Therefore
$$0 < |P_U(\ell_i)|_2 = |\ell_i|_2 \sin \theta(\ell_i,V) < |\ell_i|_2 \sin \alpha_{i,j}
\leq |\ell_i|_2 \alpha_{i,j} = \frac{1}{j},
$$
hence $0$ is a limit point of
$P_UL$ and $U \in L_{n,lim}.$
Therefore the residual set $R^\perp \subset L_{n,lim}$ so $L_{n,lim}$ is residual.
\end{proof}
\begin{theo}\label{theo2.2.11} 
$L \subset \mathbb Z^m$ is $k$-dense iff
$L_{n,dense}$ is residual.
\end{theo}
\begin{proof}
$\impliedby$ follows from Proposition \ref{prop2.2.1} since $L_{n,dense} \subset L_{n,lim}.$
\newline
$\implies$ Assume that $L$ is $k$-dense.
Then for each $p \in \mathbb Z^m$ 
Lemma \ref{2.2.2} implies that $L+p$
is $k$-dense. Define 
$R_p$ to be the residual set $R$ in 
(\ref{eq2.2.1}) but with $L$ replaced by
$L+p,$ and define the residual set
\begin{equation}\label{eq2.2.2}
	Q := Irr(n,m)^\perp \cap 
\bigcap_{p \in \mathbb Z^m} R_p.
\end{equation}
Let $U \in Q^\perp.$
Proposition \ref{prop2.2.1} implies that for every
$p \in \mathbb Z^m,$
$0$ is a limit point of 
$P_U(L-p) = P_UL- P_Up$ hence 
$P_Up$ is a limit
point of $P_UL.$ Therefore every point in
$P_U\mathbb Z^m$ is a limit point of $P_UL.$
Since $U \in Irr(n,m),$
Lemma \ref{lem1.2.2}
implies that
$P_U \mathbb Z^m$ is dense in $U.$. Since $P_UL$ is dense 
in $P_UZ$ it is dense in $U.$ This completes the proof.
\end{proof}
\subsection{Bounds on Measures of Balls}\label{sub2.3}
%
Previously we used Grassmannians
$Gr(n,m)$ and $Gr(k,m)$ and their associated $O(m)$-invariant
probability measures $\sigma_n$ and $\sigma_k.$ 
In \S2.4 we need upper bounds on $\sigma_k$-measures
of open  balls, and in \S2.5 we need upper bounds on 
$\sigma_k$-measures of 
subsets of $Gr(k,m)$ called tubular neighborhoods. 
Deriving these bounds will 
require lower bounds on open balls in $Gr(k-1,m-1)$ where
$Gr(k-1,m-1)$ is embedded in $Gr(k,m)$ in a way which is described
later. To this end we consider general Grassmannians 
$Gr(a,b)$ where
$1 \leq a < b$. 
Let $\sigma_{a,b}$ denote the 
$O(b)$-invariant probability measure on $Gr(a,b).$
The bounds that we need do not appear to be
available in Grassmannian literature, but they
proliferate in Riemannian manifold literaturure.
\newline
\\
Let $P_{a,b}$ the set of 
$P \in \mathbb R^{m \times m}$
that are symmetric $P^T = P,$ 
idempotent $P^2 = P,$
and $trace \, P = n.$
It is a real algebraic variety because
it is defined by a system
of polynomial equations. 
(\cite{warner}, p. 30, Implicit Function Theorems) implies that 
$P_{a,b}$ is a smooth
$a(b-a)$-dimensional submanifold of  
$\mathbb R^{b \times b}.$
Thus it is a real algebraic variety without singularities so it is a
real-analytic manifold.
The Frobenius scalar product on 
$\mathbb R^{b \times b}$ gives a 
scalar product $g$ on each tangent space of 
$P_{a,b}$ making it a Riemannian
manifold.
The map 
$\iota_{a,b} : Gr(a,b) \to P_{a,b}$
defined by
$\iota_{a,b}(U) := P_U$ 
is a bijective diffeomorphism (meaning that it and its inverse
are smooth). $\iota_{a,b}$ and $\iota_{a.b}^{-1}$ are
real-analytic but we will not use this fact.
The geodesic distance
$$
	d_g(P_U,P_V) = \sqrt {\theta_1^2 + \cdots + \theta_a^2}
$$
where $\theta_j$ are the canonical angles for $U,V.$
This is proved, using extensive computation, in
(\cite{wong}, p. 591, Theorem 8)
and
(\cite{keenan}, p. 28).
Engineers prefer the more computationally tractable chordal distance
$$
	|P_U-P_V|_F=  \sqrt { \sin^2 \theta_1 + \cdots + \sin^2 \theta_a}.
$$
Let $Vol_{a,b}$ denote the Riemannian measure
(also called a volume) on $P_{a,b}.$
For $g \in O(b)$ and $U \in Gr(a,b),$
$
	P_{gU} = g P_U g^{-1}.
$
The  Riemannian structure on $P_{a,b}$ is 
$O(b)$-invariant since the Frobenius scalar product on
$\mathbb R^{b \times b}$ is $O(b)$-invariant. 
Therefore the uniqueness of $O(m)$-invariant measures on
$Gr(a,b)$ up to a positive multiple implies that
$$
	Vol_{a,b}= Vol_{a,b}(Gr(a,b)) \, \sigma_{a,b}
$$
where $Gr(a,b)$ has the Riemannian structure defined by 
$\iota_{a,b}.$
For $V \in Gr(a,b)$ and $s \geq 0$ let 
$B_{geo}(V,s)$ denote the ball with geodesic radius $s$
and define
$$vol_{a,b}(s) := \sigma_{a,b}(B_{geo}(V,s)).$$
(independent of $V$ due to the transitive $O(b)$ action).
We now prove that there exist
$C_2 > C_1 > 0$
such that 
\begin{equation}\label{eq2.3.1}
C_1 \, s^{a(b-a)} \leq vol_{a,b}(s) \leq C_2 \, s^{a(b-a)}, \ \ 
0 \leq s \leq \frac{\pi\sqrt a}{2}.
\end{equation}
Note that $B_{geo}(V,\frac{\pi\sqrt a}{2}) = Gr(a,b)$
and
$\frac{\pi\sqrt a}{2}$ is the geodesic radius of $Gr(a,b).$
(\cite{gray}, p. 337, Theorem 1) derives a 
Taylor series for the Riemannian volume of 
geodesic balls with fixed center and sufficiently small 
radii on general Riemannian manifolds. This gives a Taylor series for 
$vol_{a,b}(s)$ whose
leading term equals $c\, s^{a(b-a)}$ with $c > 0.$ 
Choose $s_0 > 0$ sufficiently small so that the modulus 
of the sum of the higher order terms in the Taylor expansion
does not exceed $c\, s^{a(b-a)}/2$ when $s \leq s_o.$
Define $c_1 := c/2,$  and  $c_2 := 3c/2.$ Then
$$
c_1 \, s^{a(b-a)} \leq vol_{a,b}(s) \leq c_2 \, s^{a(b-a)}, \ \ 0 \leq s  \leq s_0,
$$
so (\ref{eq2.3.1}) holds for
$C_1 := 
c_1 \, 
\min \{1, (\frac{2s_0}{\pi \sqrt a})^{a(b-a)}\}$
and
$C_2 := 
\max\{c_2, s_0^{-a(b-a)}\}.$
\newline
\\
For $s > 0,$ $S \subset Gr(a,b)$ is called 
an $s$-packing if the geodesic balls with radius $s$ and centers
in $S$ are pairwise disjoint, and is called an $s$-cover
if these balls cover $Gr(a,b).$
The packing number 
$M(s)$ is the largest cardinality of $s$-packings and the 
covering number 
$N(s)$ is the smallest cardinality of
$s$-coverings. An $s$-packing is maximal if adding any point to it 
renders it a non $s$-packing.
An $s$-cover is minimal if removing any point makes it a non $s$-cover. 
It is easy to show that such maximal packings and minimal covers exist.
It is also easy to show that a maximal $s/2$-packing is 
an $s$-cover, and that if $S$ is an $s$-packing then
every $s$-cover has at least $|S|$ elements since
no ball of radius $s$ can contain more than one point in $S.$
This gives the packing-covering duality 
$$
M(s) \leq N(s) \leq M(s/2).
$$
Clearly
$
	M(s) \leq \frac{1}{vol_{a,b}(s)} \leq \ N(s)
$
which gives Hamming and
Gilbert-Varshamov bounds for volumes of
geodesic balls 
(\cite{keenan}, p. 60, Theorems 5.1, 5.2).
\begin{prop}\label{prop2.1.1} 
If  $\beta \in (0, \pi \sqrt a/2]$
then there exists
a $\beta$-cover $S$ of $Gr(a,b)$ satisfying
$|S| \leq C_1^{-1}\,  2^{a(b-a)} \, \beta^{-a(b-a)}\, $
where the coefficient $C_1$ appears in 
(\ref{eq2.3.1}).
\end{prop}
\begin{proof}
Let $S$ be a maximal $\beta/2$-packing hence a $\beta$-covering.
The previous inequality gives 
$
|S| \leq M(\beta/2) \leq  \frac{1}{vol_{a,b}(\beta/2)} 
$
so (\ref{eq2.3.1}) concludes the proof.
\end{proof}
\subsection{Lacunary Condition Implying
 $R$ is Null}\label{sub2.4}
%
Recall that  
$R \subset L_{n,lim} \subset Gr(k,m)$ is the residual set defined by  (\ref{eq2.2.1}).
\begin{defi}\label{def2.3.1} 
For $\rho \in \mathbb R_+,$ 
$L \subset \mathbb Z^m$ is $\rho$-lacunary if
\begin{equation}\label{eq2.3.2}
	\sum_{\ell \in L \backslash \{0\}} 
| \ell |_2^{-\rho} < \infty
\end{equation}
\end{defi}
Define 
$N(r) := |(\mathbb Z^m \backslash \{0\}) \cap B_m(0,r)|.$ 
We note that  
$\mu(B_m(0,1)) = \frac{\pi^{m/2}}{\Gamma(1 + m/2)}$
and 
$N(r) =\mu(B_m(0,1)) \, r^m + O(r^{m-1}).$
The Riemann-Stieltjes integral gives
$$
\sum_{\ell \in \mathbb Z^m \backslash \{0\}} |\ell|_2^{-\rho} = 
\int_{r = 1/2}^\infty r^{-\rho} \, dN(r) = 
\begin{cases}
\infty , \ \ \rho \leq m \\
O(\frac{1}{\rho - m}), \ \ a > m,
\end{cases}
$$
hence $\mathbb Z^m \backslash \{0\}$ is 
$\rho$-lacunary iff $\rho > m.$
For $V \in Gr(k,m)$ and $0 \leq \alpha \leq \frac{\pi}{2},$
$$A_k(V,\alpha) \subset B_{geo}(V,\alpha \sqrt k) \subset Gr(k,m).$$
Therefore (\ref{eq2.3.1}), with $a = k,$ $b = m$ and $\sigma_{a,b} = \sigma_k,$ 
implies that
\begin{equation}\label{eq2.3.3}
	\sigma_k(A_k(V,\alpha)) \leq 
C_2 \,k^{nk/2} \, \alpha^{nk}.
\end{equation}
\begin{lem}\label{lem2.3.1} 
If $L$ is $nk$-lacunary, then
$R$ is $\sigma_k$-null.
\end{lem}
\begin{proof}
For every $j \in \mathbb Z_+,$
$R \subset \bigcup_{i = 1}^\infty  
A_k(V_i, \alpha_{i,j})$ 
where 
$\alpha_{i,j} = (j\, |\ell_i|_2)^{-1}.$
Therefore
(\ref{eq2.3.3}) implies
$$
\sigma_k(R) \leq 
C_2 \, k^{nk/2} \, j^{-nk} \,
\sum_{j = 1}^\infty |\ell_j|^{-nk}.
$$
Since $L$ is $nk$-lacunary, and the sequence $\ell_j$ constructed in
Lemma \ref{lem2.2.3} has distinct terms, the second factor on the right is finite. 
Then since $j$ can be arbitrarily large,
$\sigma_k(R) = 0$ so $R$ is $\sigma_k$-null.
\end{proof}
\subsection{Lacunary Condition Implying  $L_{n,lim}$ is Null}\label{sub2.5}
Lemma \ref{lem2.2.1} 
defined subsets 
$S_{k,\ell} \subset Gr(k,m), 
\ \ \ell \in \mathbb Z^m \backslash \{0\}$
and showed they a isometric embeddings of 
$Gl(k-1,m-1)$ with respect to the 
metric $\theta_k$ so have
dimension $n(k-1)$ and codimension $n.$ 
\begin{defi}\label{def2.4.1} 
The tubular neighborhoods of $S_{k,\ell}$ are
\begin{equation}\label{eq2.4.1} 
	T(S_{k,\ell}, \alpha) := 
\bigcup_{W \in S_{k,\ell}} A_k(W,\alpha), \ \ \alpha > 0.
\end{equation}
\end{defi}
\begin{lem}\label{lem2.4.1} 
For $\ell \in \mathbb Z^m \backslash \{0\}$ and
$\alpha >0$
\begin{equation}\label{eq2.4.2} 
	T(S_{k,\ell},\alpha) = 
\{V \in Gr(k,m) : \theta(\ell, V) < \alpha\}.
\end{equation}
\end{lem}
\begin{proof} 
$V \in T(S_{k,\ell},\alpha)$ 
iff there exists $W \in Gr(k,m)$ with  
$\ell \in W$ and $V \in A_k(W,\alpha)$ 
iff there exists $W \in Gr(k,m)$ with  
$\ell \in W$ and $W \in A_k(V,\alpha)$ 
iff $\ell \in C(A_k(V,\alpha))$
iff $\theta(\ell,V) < \alpha.$
The last equivalence holds by
Lemma \ref{lem2.1.8}.
\end{proof}
\begin{cor}\label{cor2.4.1} 
$ \exists \
h_k :  (0,\pi/2] \to \mathbb R_+$ 
such that $\forall  \ \ell \in \mathbb Z^m \backslash \{0\},$
$$
		\sigma_k(T(S_{k,\ell},\alpha)) = h_k(\alpha), \ \ \alpha \in (0,\pi/2].
$$
For every $\gamma \in (0,\pi/2],$
$$
	L_{n,lim}^\perp \subset \bigcup_{\ell \in L \backslash \{0\} }
	T(S_{k,\ell}, \gamma \, |\ell|_2^{-1}).
$$
\end{cor}
\begin{proof}
If $\ell, \ell^\prime \in \mathbb Z^m \backslash \{0\},$
then there exists $g \in O(m)$ with 
$\mathbb R\ell^\prime = g(\mathbb R\ell)$ so
Since $\theta(\ell^\prime, gV) = \theta(\ell,V)$ 
for every $V \in Gr(k,m),$
Lemma \ref{lem2.4.1} implies that
$T(S_{k,\ell^\prime}, \alpha) = g(T(S_{k,\ell}, \alpha)).$
The first assertion then follows since 
$\sigma_k$ is $O(m)$-invariant.
If $U \in L_{n,lim}$ and $\gamma \in (0,\pi/2],$ then
there exists $\ell \in L \backslash \{0\}$ 
such that 
$0 < |P_U\ell|_2 < 2 \gamma/\pi.$
Let $V := U^\perp.$ 
Then
$|P_U\ell|_2 = |\ell|_2 \, \sin \theta(\ell,V).$  
and Jordan's inequality $\sin \theta \geq \
2 \theta/\pi$ on
$[0,\pi/2]$ gives
$\theta(\ell,V) \leq \gamma \, |\ell|_2^{-1}.$ 
Then Lemma \ref{lem2.4.1} implies that
$V \in T(S_{k,\ell},\gamma \, |\ell|_2^{-1}).$
This concludes the proof of the second assertion.
\end{proof}
In 1939 Hermann Weyl \cite{weyl} 
defined, for compact submanifolds $S$ 
of Euclidean space and $r > 0,$ 
$T(S,r) :=$ set of points with distance $< r$ from $S,$
He called these tubes
and proved that there exists constants
$C_S > 0$ such that 
$\mu(T(S,r)) < C_S \, r^d$ where $d$ is the
codimension of $S.$
The geodesic distance formular shows that
$S_{k,\ell}$ is a totally geodesic 
submanifolds of $Gr(k,m)$ because the geodesic
connecting points $W, V \in S_{k,\ell}$ in
$Gr(k,m)$ lies within $S_{k,\ell}.$ This
is in sharp contrast to th embedding of
$P_{k,m}$ in $\mathbb R^{m \times m}$
where geodesic distance in $P_{k,m}$ between 
distinct $P_V$ and $P_W$ is larger than
their geodesic distance in
$\mathbb R^{m \times m}$ which is their
chordal distance.
For $V \in S_{k,\ell}$ geodesic balls in 
$S_{k,\ell}$ with center $V$ equal the 
intersection of $S_{k,\ell}$ with a 
geodesic ball in $Gr(j,n)$ having
the same center and radius. We will 
exploit these properties to derive a
Weyl-type upper bound 
$h_k(\alpha) \leq D_k \, \alpha^n$ 
in the proof of the following result.
\begin{theo}\label{thm2.4.1} 
If $L$ is $n$-lacunary, then  $L_{n,lim}$ is 
$\sigma_n$-null.
\end{theo}
\begin{proof}
It suffices to construct the Weyl-type bound 
$h_k(\alpha) \leq D_k \, \alpha^n$
because then Corollary \ref{cor2.4.1}
implies that for every $\gamma > 0,$
$$
\sigma_k(L_{n,lim}^\perp) \leq
D_k \gamma^n \, 
\sum_{\ell \in L \backslash \{0\}}
|\ell|_2^{-n}.
$$ 
Since $L$ is $n$-lacunary,
the sum is finite, hence since $\gamma > 0$
is arbitrary $L_{n,lim}^\perp$ is $\sigma_k$-null so
Lemma  \ref{lem2.1.1} implies that $L_{n,lim}$ 
is $\sigma_n$-null.
For $k = 1,$
$S_{1,\ell} = \{\mathbb R \ell\},$
$T(S_{1,\ell}, \alpha) = A_1(\mathbb R\ell, \alpha)$
and (\ref{eq2.3.3}) implies that
$D_1 := C_2$ gives  such a bound.
So we assume that $k \geq 2.$
Proposition \ref{prop2.1.1} with
$a = k-1$ and $b = m-1$ 
implies that for 
$\beta \in (0,\pi \sqrt {k-1} \, / 2]$
there exists a $\beta$-cover $S$ of $S_{k,\ell}$
with
$$|S| \leq C_1^{-1} \, 2^{n(k-1)} \, \beta^{-n(k-1)}.$$ 
This means
$S_{k,\ell} \subset \bigcup_{V \in S} B_{geo}(V,\beta)$
with geodesic balls in $Gr(k,m).$
Since $B_{geo}(V,\beta) \subset A_k(V,\beta)$ 
the triangle inequality gives
$$T(S_{k,\ell},\alpha) \subset
\bigcup_{V \in S}  A_k(V,\alpha + \beta).$$
Combining the bound on $|S|$
and  (\ref{eq2.3.3}) gives 
$$
h_k(\alpha) 
\leq 
C_1^{-1} C_2 \, k^{nk/2}\, 2^{n(k-1)}\, \beta^{-n(k-1)}(\alpha + \beta)^{nk}\}, \ \ 
\beta \in (0,\pi \sqrt {k-1} \, /2].
$$
Recall that $k \geq 2.$ Define
$\alpha_0 := \pi/(2\sqrt {k-1}).$
For $\alpha \in [0,\alpha_0]$ the choice
$\beta = (k-1)\alpha$ lies in the range
$[0,\pi \sqrt{k-1}\, /2 ]$ and minimizes the bound above to give 
$$
	h_k(\alpha) \leq \widetilde D_k \, \alpha^n, \ \ \alpha \in [0,\alpha_0]
$$
where
$$
\widetilde D_k := C_1^{-1}\,  C_2 \, k^{3nk/2}\, \left(\frac{2}{k-1}\right)^{n(k-1)}.
$$
Since $\sigma_k$ is a probability measure
$h_k \leq 1,$ hence
$$
	h_k(\alpha) \leq D_k \, \alpha^n, \ \ \alpha \in [0,\pi/2],
$$
where 
$D_k := \max \{\widetilde D_k, /\alpha_0^{-n}\}.$
\end{proof}
We compare \S2.3 and \S2.4. 
Since $n$-lacunary implies $kn$-lacunary and 
the residual subset $R \subset Gr(k,m)$ 
in (\ref{eq2.2.1})
satisfies $R^\perp \subset L_{n,lim},$ 
Theorem \ref{thm2.4.1} implies
Lemma \ref{lem2.3.1} 
when $k = 1.$ For $k \geq 2$ 
it trades a stronger hypothesis
for a much stronger conclusion.
\section{Diophantine Approximation on Grassmannians}\label{sec3} 
Recall the definition of $L_{n,\psi}$ in \S1.2.  
If $\psi$ converges to $0$ then 
$L_{n,\psi} \subset L_{n,lim},$ but 
the reverse inclusion does not 
necessarily hold since $L_{n,\psi}$
depends on the rate of convergence of 
$\psi.$  Let $\mu$ denote Lebesgue measure on 
the Euclidean space 
$\mathbb R^{n \times k}$ equipped with the Frobenius scalar product.
We will use a classic result in 
Diophantine  approximation to prove that
for certain $\psi,$
$Z^m_{n,\psi} $ is $\sigma_n$-null hence 
$L_{n,\psi}$ is $\sigma_n$-null for every $L \subset \mathbb Z^m.$
\subsection{Manifold Structure of $Gr(n,m)$}\label{sub3.1}
%
To study Diophantine approximation 
we need local coordinates on $Gr(n,m).$
Let $b := \binom{m}{n},$
$[m] := \{1,...,m\},$ and
$\mathcal I_n := 
\{I \subset [m] : |I| = n\},$
so $|\mathcal I_n| = b.$
For $I \in \mathcal I_n$ define
$I^c := [m] \backslash I \in \mathcal I_k.$ 
For $I \in \mathcal I_n,$ $j \in \mathbb Z_+,$
and $M \in \mathbb R^{m \times j},$
let
$M_I \in \mathbb R^{n \times j}$
denote the submatrix of $M$ having row numbers in $I,$ 
and
$M_{I^c} \in \mathbb R^{k \times j}$ 
denote the submatrix of $M$ having row numbers in $I^c.$
Clearly $M_I$ and $M_{I^c}$ determine $M.$
Every $U \in Gr(n,m)$ admits a basis with $n$ elements so has
a representation $U = M\mathbb R^n$ with basis vectors the columns of $M \in \mathbb R^{m \times n}.$  Clearly
$rank M = n.$ 
For $M \in \mathbb R^{m \times n},$ 
the numbers 
$\det M_I, \, I \in \mathcal I_n$ 
are called the minors of $M.$ 
For $I \in \mathcal I_n$ define
\begin{equation}\label{eq3.1.1} 
	\mathcal U_I := \{U = M\mathbb R^n : 
M \in \mathbb R^{m \times n}, \, \det M_I \neq 0\}.
\end{equation}
These are well defined since
$N\mathbb R^n = M\mathbb R^n$ iff
$\exists  A \in GL(n,\mathbb R)$ with
$N = MA,$ so $\det N_I = \det M_I \det A \neq 0.$ They cover $Gr(n,m)$ since $rank M = n$ iff $M$ has a nonzero minor.
Define 
$\Phi_I : \mathcal U_I \to 
\mathbb R^{n \times k}$  by
\begin{equation}\label{eq3.1.2} 
	\Phi_I(M\mathbb R^n) := 
 \, (M_{I^c}M_I^{-1})^T, 
\ \ \det M_I \neq 0.
\end{equation}
This is well defined because if $N = MA$ where
$A \in GL(n,\mathbb R)$ then
$N_{I^c}N_I^{-1} = 
(MA)_{I^c}(MA)_I^{-1} = 
M_{I^c}A (M_IA)^{-1} =
M_{I^c}AA^{-1} M_I^{-1} = 
M_{I^c}M_I^{-1}.$
\newline
\\
$\Phi_I$ is injective because if
$\Phi_I(M\mathbb \mathbb R^n) = 
\Phi_I(N \mathbb R^n)$
then
$(MM_I^{-1})_{I^c} = 
M_{I^c}M_I^{-1} = 
N_{I^c}N_I^{-1} =
(NN_I^{-1})_{I^c}$
and
$(MM_I^{-1})_I =
M_IM_I^{-1} = 
I_n = 
N_IN_I^{-1} = 
(NN_I^{-1})_I$
hence
$MM_I^{-1} =
NN_I^{-1}$
so
$N\mathbb R^n = M\mathbb R^n.$
\newline
\\
$\Phi_I$ is surjective because if
$A \in \mathbb R^{n \times k}$
then define 
$B \in \mathbb R^{m \times n}$ 
so that
$B_I := I_n$
and
$B_{I^c} := A^T.$
Then 
$\Phi_I(B\mathbb R^n) = 
(B_{I^c} B_I^{-1})^T = 
(A^T I_n^{-1})^T =A.$ 
\newline
\\
We use these bijections to
construct the same Euclidean manifold structure on each
set $\mathcal U_I$ as on $\mathbb R^{n \times k}.$To show that these define a
differential manifold structure in $Gr(n,m)$ we must show that for any $I, J \in \mathcal I_n,$ the transition functions
$
\Phi_J \circ \Phi_I^{-1} : 
\Phi_I(\mathcal U_I \cap \mathcal U_J) \to
\mathbb R^{n \times k}
$
are smooth. If 
$U \in \mathcal U_I \cap \mathcal U_J$
then there exists $M \in \mathbb R^{m \times n}$
with $M_I= I_n,$ $\det M_J \neq 0,$
and $U = M\mathbb R^n.$
Then
$$
\Phi_J \circ \Phi_I^{-1}(\Phi_I(U))
=
\Phi_J \circ \Phi_I^{-1}((M_{I^c})^T) 
=
\Phi_J(M\mathbb R^n)
= 
(M_{J^c}M_J^{-1})^T..
$$
The entries of $M$ are the entries of 
$(M_{I^c})^T$ together with entries of
$I_n$ and the last expression is a real-analytic function of $M$ since
$\det M_J \neq 0.$
This proves that the transition functions are
real-analytic and hence smooth.
This gives a real-analytic manifold structure to $Gr(n,m)$ and this is stronger 
than a smooth structure. 
\begin{prop}\label{prop2.3.1}
For $I \in \mathcal I_n$ 
the bijection 
$\Phi_I :  \mathcal U_I \to \mathbb R^{n \times k}$ is a diffeomorphism 
with respect to the Riemannian structure on $\mathcal U_I$ that makes
the map $U \mapsto P_U$  a diffeomorphism
of $\mathcal U_I$  into $P_{n,m}.$
$\Phi_I$ gives a bijection between 
$\sigma_n$-null subsets of 
$\mathcal U_I$ and 
$\mu$-null subsets of
$\mathbb R^{n \times  k}$
where $\mu$
is Lebesgue measure.
\end{prop}
\begin{proof}
For $I \in \mathcal I_n$ define
$\mathcal B_I := 
\{ B \in \mathbb R^{m \times n} : B_I = I_n\}$
equipped with the Euclidean structure by applying the Frobenius norm to 
$B_{I^c}^T \in  \in \mathbb R^{n \times k}.$
Clearly $B \mapsto B\mathbb R^n$ is a  
bijection from $\mathcal B_I$ 
onto $\mathcal U_I.$
Moreover 
$\Phi_I(B\mathbb R^n) = B_{I^c}^T$
is a diffeomorphism from 
$\mathcal B_I$ onto  $\mathbb R^{n \times k}$
because it is an isometry between these two  Euclidean spaces. In order to show that 
$\Phi_I$ is a diffeomorphism
with respect to the Riemannian 
structure defined by the
embedding 
$U \mapsto P_U,$ 
we must show that the  map 
$P_{B\mathbb R^n} \mapsto B$  is a diffeomorphism.
Clearly $P_{B\mathbb R^n} = B(B^TB)^{-1}B^T.$
Computing $B$ from this highly nonlinear function of
$B$ appears daunting. So consider the inverse function
$F_I : \mathcal B_I \to P_{n,m}$
defined by 
$F_I(B) := B(B^TB)^{-1}B^T.$
This is clearly a real-analytic and hence a smooth
homeomorphism from $\mathcal B_I$ onto its
image under $F_I.$ These properties do not
force $F_I$ to be a diffeomorphism. The map
$x \mapsto x^3$ is a real-analytic homeomorphism
of $\mathbb R$ onto $\mathbb R$ but is not 
a diffeomorphism. The differential 
$d(x^3) = (3x^2)dx$ is a linear function of the
differential $dx$ but this linear function has
rank $0$ at $x = 0,$ the point where the inverse function is not differentiable.
The differential $dF_I(B)$ is a linear function (which may depend on $B$)
of the differential $dB.$ Its rank is defined to be the rank of this
linear function (so the rank nay depend on $B$). 
The powerful 
(\cite{warner}, p. 23, 1.30 Inverse Function Theorem) implies that 
if $F$ is a smooth injection then $F$ is a diffeomorphism 
iff the rank of $dF(B)$ equals the dimension $nk$ of 
$\mathcal B_I$ for every $B \in \mathcal B_I.$
Now we use matrix calculus to compute
$$dF(B) = (dB)(B^TB)^{-1}B^T + B[d(B^TB)^{-1}]B^T + 
B(B^TB)^{-1}(dB^T).$$
where
$d(B^TB)^{-1} = -(B^TB)^{-1} [(dB)^TB + B^T(dB)](B^TB)^{-1}.$
Since $Gr(n,m)$ has a transitive $O(m)$ action the rank is independent of the choice of $I$ and $B.$
So we choose $I = \{1,...,n\}$ and $B = [e_1 \cdots e_n].$ 
Then we have the blocked forms 
$B = [I_n ; 0], dB = [0 ; (dB)_{I^c}],$ 
hence
$dF_I(B) =  
[0, (dB)_{I^c}^T ; (dB)_{I^c} , 0].$
Then $dF_I(B)$ has rank $nk$ so 
$F_I,$ and hence
$\Phi_I,$ are diffeomorphisms. 
\newline
We can choose local coordinates $x_1,...,x_d, d = nk$ 
are the entries of $\Phi_I.$ 
Then the Riemannian measure of any Borel 
$S \subset Gr(n,m)$  
equals the integral of the differential $d$-form 
$\omega_g$ form over$S.$ where
(\cite{lee}, p. 388, The Riemannian Volume Form) 
shows that
$$\omega_g = \sqrt {\det\, (g_{i,j}) } \,
dx_1 \wedge \cdots \wedge dx_{nk}.
$$
where $g$ is the metric tensor and 
$$g_{i,j} = 
g(\frac{\partial}{\partial x_i}, \frac{\partial}{\partial x_j}).
$$
The second assertion follows since $g_{i,j}$
is positive definite and 
$dx_1 \wedge \cdots \wedge dx_d$ represents 
Lebesgue measure on $\mathbb R^{n \times k}.$ 
We use the fact that $\sigma_n$  is a 
positive multiple of Riemannian measure.
\end{proof}
\subsection{Collection of Compact Sets that Covers $Gr(n,m)$}\label{sub3.2}
%
The Stiefel manifold $St(n,m),$ 
consists of 
$M \in \mathbb R^{m \times n}$ satisfying 
$F(M) = 0$ wher
$F(M) := M^TM - I_n.$
Since $F$ is a smooth function on $\mathbb R^{m \times n},$
(\cite{warner}, p. 30, Implicit Function Theorems) implies that
$St(n,m)$ is a smooth manifold if the differential
$dF = (dM)^TM + M^T(dM)$ has constant rank and 
then the dimension equals $nm$ minus this rank.
Since $O(m)$ act transitively on $St(n,m)$ by left multiplication
and $dF$ is invariant under this action, the rank of $dF$ is the same
for all $M \in St(n,m).$ Choose $M := [e_1 \cdots e_n]$
and write $dM = [E_1 ; E_2]$ where
$E_1 \in \mathbb R^{n \times n},$ 
$E_2 \in \mathbb R^{k \times n}$
are given by, respectively, 
the first $n$ and last $k$ rows of $dM.$
Then $dF = E_1^T + E_1$ consists of 
symmetric $n$ by $n$ matrices 
so the rank of
$dF$ equals $n(n+1)/2$ and the dimension
of $St(n,m)$ equals $nm - n(n+1)/2.$
Since the entries of $F$ are polynomials
$St(n,m)$ is a real algebraic variety.
Since it is a smooth algebraic variety it
has no singularities hence is a real 
analytic manifold.
\newline
\\
In this section we use $st(n,m)$ to construct 
a finite covering of
$Gr(n,m)$ by compact subsets.
The map 
$M \mapsto M\mathbb R^n$
defines a surjection of $St(n,m)$ onto
$Gr(n,m)$ 
since every $n$-plane 
has an orthonormal basis. Moreover
$M_2\mathbb R^n= M_1\mathbb R^n$
iff there exists $O \in O(n)$ with
$M_2 = M_1O.$
Define $\beta := 1/\sqrt b$
and for every $I \in \mathcal I_n$ define
$St_{I,\beta}(n,m) := \{M \in St(n,m) : |\det M_I| \geq \beta\}$
and
\begin{equation}\label{eq3.2.1} 
	\mathcal U_{I,\beta} := 
St_{I,\beta}(n,m)\mathbb R^n.
\end{equation}
\begin{lem}\label{lem3.2.1}
The sets $\mathcal U_{I,\beta}$ 
are compact and cover $Gr(n,m).$ 
There exists a compact 
$C \subset \mathbb R^{n \times k}$ 
such that for every
$I \in \mathcal I_n,$
$\Phi_I(\mathcal U_{I,\beta}) = C.$
\end{lem}
\begin{proof}
Assume $U \in \mathcal U_{I,\beta}.$
Then there exists $M \in St_{I,\beta}(n,m)$ 
with $U = \varphi(M).$
Moreover
$|M_I|_2 \leq 1$
and the eigenvalues of
$M_I^T M_I$ are in the interval
$[\beta^2, 1]$ hence
$|M_I^{-1}|_2 \leq \beta^{-1}.$
Therefore
if $\Phi_I(U) 
= (M_{I^c}M_I^{-1})^T$ satisfies
\begin{equation}\label{eq3.2.2} 
	|\Phi_I(U)|_2 \leq \beta^{-1},
\end{equation}
so $\Phi_I(\mathcal U_{I,\beta})$ is a compact subset of $\mathbb R^{n \times k}.$ Since the manifold structure on
$\mathcal U_I$ was constructed to equal the Euclidean manifold structure on 
$\mathbb R^{n \times k},$
$\mathcal U_{I,\beta}$ is a 
compact subset of $\mathcal U_I.$
For $M \in St(n,m)$ 
the Cauchy-Binet formula 
(\cite{horn}, p. 28, Section 0.8)
gives
\begin{equation}\label{eq3.2.3} 
	1 = \det I_n = \det (M^TM) = 
\sum_{I \in \mathcal I_n} (\det M_I)^2.
\end{equation}
Therefore at least one minor of $M$ has modulus $\geq \beta$ so the sets
$\mathcal U_{I,\beta}$ cover $Gr(n,m).$
To prove the last assertion for every
$I, J \in \mathcal I_n$ define the permutation matrix 
$P_{I,J} \in O(m)$ by
$(P_{I,J})_I = (I_m)_J$
and 
$(P_{I,J})_{I^c} = (I_m)_{J^c}.$
Then
$P_{I,J}St_{J,\beta}(n,m) = St_{I,\beta}(n,m)$ 
and a computation gives
\begin{equation}\label{eq3.2.4} 
\begin{array}{lcccccc}
	\Phi_I(\mathcal U_{I,\beta})  & = & 
 \Phi_I(St_{I,\beta}(n,m)\mathbb R^n) & = &
\Phi_I(P_{I,J}St_{J,\beta}(n,m)\mathbb R^n) & = & \\
\Phi_J(St_{J,\beta}(n,m)\mathbb R^n) & = &
\Phi_J(\mathcal U_{J,\beta}). 
\end{array}
\end{equation}
\end{proof}
\begin{lem}\label{lem3.2.2}
If $I \in \mathcal I_n,$ 
$U \in \mathcal U_{I,\beta},$
and
$\ell \in \mathbb Z^m,$ then
\begin{equation}\label{eq3.2.5} 
|\Phi_I(U)\ell_{I^c} + \ell_I|_\infty \leq
\beta^{-1} \sqrt {m}\, |P_U\ell|_\infty
\end{equation}
and
\begin{equation}\label{eq3.2.6} 
|\ell|_\infty \leq \beta^{-1}\left(\sqrt k |\ell_{I^c}|_\infty + 
 \sqrt m |P_U\ell|_\infty \right).
\end{equation}
\end{lem}
\begin{proof}
Choose $M \in St_{I,\beta}(n,m)$ with $U = M\mathbb R^n,$ multiply the identity
$M^T\,\ell  = 
M_{I^c}^T \, \ell_{I^c} + M_I^T \, \ell_l.$
by $(M_I^{-I})^T,$ 
then substitute $M^T = M^TP_U$ and $\Phi_I(U)$ in 
(\ref{eq3.1.2}) to obtain
$$(MM_I^{-1})^T P_U\ell = \Phi_I(U) \, \ell_{I^c} +\ell_I.$$ (\ref{eq3.2.5}) then follows since 
$$
|(MM_I^{-1})^T P_U\ell|_\infty \leq 
\beta^{-1} |P_U\ell|_2 \leq 
\beta^{-1} \sqrt {m}\, |P_U\ell|_\infty.
$$
The inequality (\ref{eq3.2.6}) with
$\ell$ replaced by $\ell_I$ 
follows from the triangle inequality
and the fact that
$$
|\Phi_I(U)\ell_{I^c}|_\infty 
\leq \beta^{-1} |\ell_{I^c}|_2 
\leq \beta^{-1} \sqrt k \,|\ell_{I^c}|_\infty.
$$
Then  (\ref{eq3.2.6}) follows since
$|\ell|_\infty = 
\max\{|\ell_{I^c}|_\infty, |\ell_I|_\infty\}$
and $\beta^{-1}\sqrt {k} \geq 1.$
\end{proof}
\subsection{The Khintchine-Groshev Theorem in $Gr(n,m)$}\label{sub3.3}
%
For $\psi : \mathbb Z_+ \to \mathbb R_+$
let $\mathcal A_{n,k}(\psi)$ denote the subset
of $A \in [0,1)^{n \times k}$ such that
\begin{equation}\label{eq3.3.1}
	|Aq+p|_\infty < \psi(| \, q|_\infty \, )
\end{equation}
holds for infinitely many pairs
$(p,q) \in \mathbb Z^n \times \mathbb Z^k.$
Let $\mu$ denote Lebesgue measure on 
the Euclideaan space 
$\mathcal \mathbb R^{n \times k},$
and $\mathcal B_{n,k}$ denote the Banach space of sequences 
$\psi : \mathbb Z_+ \to \mathbb R$
with finite norm 
$||\psi||_{n,k}$ where
$||\psi||_{n,k}^n := 
\sum_{j \in \mathbb Z_+} j^{k-1} \, |\psi(j)|^n.$
Clearly if $\psi \in \mathcal B_{n,k},$ then
$\psi$ converges to $0$ hence
$\psi \leq \Psi := \max \{|\psi(j)| : j \in \mathbb Z_+\}.$
The Khintchine-Groshev theorem (\cite{beresnevich}, p. 2)  states
\begin{equation}\label{eq3.3.2}
\mu(\mathcal A_{n,k}(\psi)) = 
\begin{cases}
  0 \text{ if } \psi \in \mathcal B_{n,k} \\
  1 \text{ if } \psi \notin \mathcal B_{n,k}
\text{ and } \psi \text{ is non-increasing}.
\end{cases}
\end{equation}
This theorem was proved for $n = 1$ in 1926 by Khintchine \cite{khintvhine},
and for $n \geq 2$ in 1938 by Groshev \cite{groshev}.
An excellent survey is given by Hauke \cite{hauke}.
We replaced the closed cube $[0,1]^{n \times k},$ used in \cite{beresnevich}, by
$[0,1)^{n \times k}$ 
because the set of its translates by
$\zeta \in \mathbb Z^{n \times k}$ partitions 
$\mathbb R^{n \times k}.$ 
Our replacement is inconsequential because
$\mu([0,1]^{n \times k} \backslash  [0,1)^{n \times k})) = 0.$
\begin{theo}\label{thm3.3.1} 
If $\psi$ is non-increasing and $\psi \in \mathcal B_{n,k},$
then $\mathbb Z^m_{n,\psi}$ is $\sigma_n$-null and
$L_{n,\psi}$ is null for every $L \subset \mathbb Z^m.$ 
\end{theo}
\begin{proof}
If $\mathbb Z^m_{n,\psi}$ is $\sigma_n$-null  then
$L_{n,\psi}$ is null for every $L \subset \mathbb Z^m$ 
because $L_{n,\psi} \subset \mathbb Z^m_{n,\psi}.$ 
So we assume that $\psi$ is non-increasing and 
$\psi \in \mathcal B_{n,k}$ and proceed to show that
$\sigma_n(\mathbb Z^m_{n,\psi}) = 0.$
Choose $I \in \mathcal I_n.$ 
By Lemma \ref{lem3.2.1} it suffices to prove that
$\sigma_n(\mathbb Z^m_{n,\psi} \cap \mathcal U_{I,\beta})  = 0.$
By Proposition \ref{prop2.3.1} it suffices to prove that
\begin{equation}\label{eq3.3.3}
\mu(\Phi_I(\mathbb Z^m_{n,\psi} \cap \mathcal U_{I,\beta}))  = 0.
\end{equation}
If $U \in \mathbb Z^m_{n,\psi} \cap \mathcal U_{I,\beta}$ then
there exists an infinite set of
$\ell \in \mathbb Z^m$ such that
$0 < |P_U\ell|_\infty < \psi(|\ell|_\infty).$
(\ref{eq3.2.6}) implies that there exists an infinite number 
of $\ell \in \mathbb Z^m$ satisfying
$$
|\ell_{I^c}|_\infty \geq 
\frac{\beta}{\sqrt k} |\ell|_\infty - \Psi \sqrt {\frac{m}{k}}.
$$
Infinitely many distinct $\ell$ forces 
$\ell \to \infty$ along a  subsequence, 
so the $\ell_{I^c}$ are infinite in number.
Then (\ref{eq3.2.5}), $|\ell_{I^c}|_\infty \leq |\ell|_\infty,$ 
and non-increasing $\psi$ give
\begin{equation}\label{eq3.3.4}
|\Phi_I(U)\ell_{I^c} + \ell_I|_\infty <
\beta^{-1} \sqrt {m}\, \psi(|\ell_{I^c}|_\infty) 
\text{ for infinitely many } \ell_{I^c}.
\end{equation}
Since 
$\Phi_I(\mathbb Z^m_{n,\psi} \cap \mathcal U_{I,\beta}) \subset C $
where 
$C \subset \mathbb R^{n \times k}$ 
is the compact set  in Lemma \ref{lem3.2.1}, 
and the sets
$[0,1)^{n \times k} + \zeta, \ \zeta \in \mathbb Z^{n \times k}$
partition $\mathbb R^{n \times k},$ 
there exists a finite 
$\mathcal F \subset \mathbb Z^{n \times k}$ 
and
$\mathcal A_{\zeta} \subset [0,1)^{n \times k}, \zeta \in \mathcal F$ 
such that
\begin{equation}\label{eq3.3.5}
\Phi_I(L_{n,\psi} \cap \mathcal U_{I,\beta}) =
\bigcup_{\zeta \in \mathcal F} \left(\mathcal A_\zeta + \zeta \right).
\end{equation}
It suffices to prove that 
$\mu(\mathcal A_\zeta) = 0$ 
for every $\zeta \in \mathcal F.$
If $A \in \mathcal A_\zeta$ then
$$
A = \Phi_I(U)  - \zeta 
\text{ where } 
U \in \mathbb Z^m_{n,\psi} \cap \mathcal U_{I,\beta}
$$
so (\ref{eq3.3.4}) implies that
\begin{equation}\label{eq3.3.6}
|A \, q + p |_\infty <  \beta^{-1} \sqrt {m}\, \psi(|q|_\infty).
\end{equation}
where $q = \ell_{I^c} \in \mathbb Z^k$ and
$p = \ell_I + \zeta \, \ell_{I^c}.$
Since (\ref{eq3.3.6}) holds for an infinite number of 
$q$ and therefore for an infinite number of  pairs
$(p,q) \in \mathbb Z^n \times \mathbb Z^k,$ 
and $\beta^{-1} \sqrt {m}\, \psi \in \mathcal B_{n,k},$ 
the Khintchine-Groshev theorem  implies that 
$\mu(A_\zeta) = 0$ and concludes the proof.
\end{proof}
\begin{conj}\label{con3.3.1} 
$L_{n,\psi}$ is conull if 
$L = \mathbb Z^m,$ $\psi \notin \mathcal B_{n,k},$ and $\psi$ is nonincreasing.
\end{conj}
\noindent {\bf Acknowledgements} This work is supported by the Krasnoyarsk Mathematical Center and financed by the 
Ministry of Science and Higher Education of the Russian Federation (Agreement No. 075-02-2026-1314). The author thanks Lior Alon and Mario Kummer for many discussions about Fourier quasicrystals, the  anonymous referee for discovering numerous serious errors and suggesting corrections, and his wife for her meticulous editing.

\end{document}